\documentclass[a4paper,10pt, final]{amsart}
\usepackage[utf8]{inputenc}
\usepackage{amsmath, amsthm, amsfonts,amssymb,epsfig,enumerate, tikz, showkeys, fancyhdr , color, tensor}
\usetikzlibrary{arrows}
\usetikzlibrary{matrix}
\usetikzlibrary{patterns}
\usetikzlibrary{decorations.markings}
\usetikzlibrary{decorations.pathmorphing}
\newcounter{res}[section]
\numberwithin{res}{section}
\newtheorem{thm}[res]{Theorem}
\newtheorem{lem}[res]{Lemma}
\newtheorem{prop}[res]{Proposition}
\newtheorem{cor}[res]{Corollary}
\newtheorem*{mainthm}{Main theorem}
\theoremstyle{definition}
\newtheorem{notation}[res]{Notation}
\newtheorem{dfn}[res]{Definition}
\newtheorem{req}[res]{Remark}

\renewcommand{\S}{\ensuremath{\mathbb{S}}}

\newcommand{\id}{\ensuremath{\mathrm{id}}}
 
\newcommand{\ZZ}{\mathbb Z} 
\newcommand{\CC}{\mathbb C} 
\newcommand{\QQ}{\mathbb Q}
\newcommand{\RR}{\mathbb R} 

\newcommand{\F}{\mathcal{F}}

\newcommand{\tr}{\mathop{\mathrm{Tr}}\nolimits}

\newcommand{\cjg}[1]{\overline{#1}}

\renewcommand {\epsilon}{\varepsilon}

\renewcommand {\leq}{\leqslant}

\renewcommand {\bar}{\cjg}

\renewcommand\marginpar[1]{}
\newcommand\eqdef{\ensuremath{\stackrel{\textrm{def}}{=}}}
\newcommand\kup[1]{\ensuremath\left\langle #1 \right\rangle}

\newcommand{\col}{\ensuremath{\mathrm{col}}}
\newcommand{\NE}{\ensuremath{\textit{NE}}}
\newcommand{\imagesfolder}{.}
\newcommand{\ie}{i.~e.~}
\newcommand{\eg}{e.~g.~}
\newcommand{\resp}{resp.{} }
\newcommand{\FR}[1]{\textit{\textbf{FR}#1}}
\newcommand{\FRp}[1]{\textit{\textbf{FR'}#1}}

\newcommand{\websquare}[1][0.6]{\vcenter{\hbox{\!
 \begin{tikzpicture}[scale={#1},decoration={markings, mark=at
     position 0.5 with {\arrow{>}}},postaction={decorate}]
      \draw[postaction = {decorate}] (0.5,0.5)--(1,1);
      \draw[postaction = {decorate}] (2.5,2.5)--(2,2);
      \draw[postaction = {decorate}] (2,1)--(1,1);
      \draw[postaction = {decorate}] (2,1)--(2,2);
      \draw[postaction = {decorate}] (2,1)--(2.5,0.5);
      \draw[postaction = {decorate}] (1,2)--(2,2);
      \draw[postaction = {decorate}] (1,2)--(1,1);
      \draw[postaction = {decorate}] (1,2)--(0.5,2.5);
 \end{tikzpicture}}}}
\newcommand{\webtwovert}[1][0.6]{
\vcenter{\hbox{\!
 \begin{tikzpicture}[scale={#1},decoration={markings, mark=at
     position 0.5 with {\arrow{>}}},postaction={decorate}]
      \draw[postaction= {decorate}] (0.5,0.5) .. controls (1,1) and (1,2) .. (0.5, 2.5);
      \draw[postaction= {decorate}] (2.5,2.5) .. controls (2,2) and (2,1) .. (2.5, 0.5);
 \end{tikzpicture}}}}

\newcommand{\webtwohori}[1][0.6]{\vcenter{\hbox{\!
 \begin{tikzpicture}[scale={#1},decoration={markings, mark=at
     position 0.5 with {\arrow{>}}},postaction={decorate}]
      \draw[postaction= {decorate}] (0.5,0.5) .. controls (1,1) and (2,1) .. (2.5, 0.5);
      \draw[postaction= {decorate}] (2.5,2.5) .. controls (2,2) and (1,2) .. (0.5, 2.5);
 \end{tikzpicture}}}}

\newcommand{\webvert}[1][0.6]{\vcenter{\hbox{\!
 \begin{tikzpicture}[scale={#1},decoration={markings, mark=at
     position 0.5 with {\arrow{>}}},postaction={decorate}]
      \draw[postaction= {decorate}] (0.5,0.5) -- (0.5, 2.5);
  \end{tikzpicture}}}}

\newcommand{\webbigon}[1][0.6]{\vcenter{\hbox{\!
 \begin{tikzpicture}[scale={#1},decoration={markings, mark=at
     position 0.5 with {\arrow{>}}},postaction={decorate}]
      \draw[postaction= {decorate}] (0.5,0.5) -- (0.5, 1);
      \draw[postaction= {decorate}] (0.5,2) -- (0.5, 2.5);
      \draw[postaction= {decorate}] (0.5,2) ..controls (0,1.5) and (0,1.5).. (0.5, 1);
      \draw[postaction= {decorate}] (0.5,2) ..controls (1,1.5) and (1,1.5).. (0.5, 1);
  \end{tikzpicture}}}}

\newcommand{\webcircle}[1][0.6]{\vcenter{\hbox{\!
 \begin{tikzpicture}[scale={#1},decoration={markings, mark=at
     position 0.5 with {\arrow{>}}},postaction={decorate}]
      \draw[postaction= {decorate}] (1.5,1.5) circle (1cm);
 \end{tikzpicture}}}}
\newcommand{\webcirclereverse}[1][0.6]{\vcenter{\hbox{\!
 \begin{tikzpicture}[scale={#1},decoration={markings, mark=at
     position 0.5 with {\arrow{<}}},postaction={decorate}]
      \draw[postaction= {decorate}] (1.5,1.5) circle (1cm);
 \end{tikzpicture}}}}

\title{Grothendieck groups of the Khovanov-Kuperberg algebras}
\author{Louis-Hadrien Robert} \thanks{The author is supported by a grant from the LabEx IRMIA} 
\address{IRMA \\
7, rue René Descartes \\
67084 Strasbourg Cedex \\
France \\}
\email{robert@math.unistra.fr}
\date{\today}
\newcommand{\sll}{\ensuremath{\mathfrak{sl}}}
\begin{document}
\maketitle
\setcounter{tocdepth}{1}
\tableofcontents
\section*{Introduction}
The aim of this paper is to take benefit of the foam nature of the Khovanov-Kuperberg algebras $K^\epsilon$ to compute the Grothendieck groups of their categories of finitely generated projective modules. 

The Khovanov-Kuperberg algebras appears when one wants to extend the $\sll_3$ link homology to tangles. It has been observed (\cite{MR1684195, 2012arXiv1206.2118M, LHR1, LHR2} that, contrasting with the $\sll_2$ case and its associated algebras $H^n$, the web bases do not correspond to the indecomposable projective modules. However, we show the following result which strengthen a little a theorem from Mackaay, Tubbenhauer and Pan \cite{2012arXiv1206.2118M}:
\begin{mainthm}
Let $W^\epsilon$ be the $\ZZ[q,q^{-1}]$-module spanned by $\epsilon$-webs subjected to the Kuperberg relations, then the map 
\[
\begin{array}{rcl}
\phi:\,\,{W^\epsilon} &\longrightarrow & K_0(K^\epsilon_{\QQ}\mathsf{\textrm{-}proj}_{\mathrm{gr}}) \\
  w &\longmapsto & [P_w].  
\end{array}
\]
is an isomorphism of $\ZZ[q,q^{-1}]$-module.
\end{mainthm}

\subsubsection*{Organization of the proof and of  the paper}
\label{sec:organisation-paper}

The map $\phi$ in the main theorem is very natural, and it follows quite immediately from the categorification of the Kuperberg bracket that it is injective. What remains to show is the surjectivity of $\phi$. The idea is to compose this map $\phi$ with an injective map (the so-called Hattori-Stallings trace see section \ref{sec:hatt-stattl-trac}) and to evaluate the dimension of the co-domain via geometrical tools (see section \ref{sec:foams-traces}). In the meanwhile, the grading has been forgotten, this causes no problem, as explained in section~\ref{sec:some-facts-about}. This will be enough to compute the Grothendieck groups (see section~\ref{sec:counting-dimension}) but not to show that the web-modules associated with non-elliptic webs give bases for them.  To show this last point, we'll need a small detour by Gornik's deformation (see section~\ref{sec:filtered-version}) of Khovanov's TQFT. 

\subsubsection*{Acknowledgments} This paper would not exist without the bright ideas of Mikhail Khovanov. The author thanks him for his warm invitation to Columbia University.

\section{The Khovanov-Kuperberg algebras}

\subsection{The decategorified picture}
\label{sec:khov-kuperb-algebr}

\subsubsection{The Hopf algebra $U_q(\sll_3)$}
\label{sec:u_qsll_3}
A more detailed introduction is given the author's PhD thesis \cite{LHRThese}.
\begin{dfn}
  The algebra $U_q(\sll_3)$ is the associative $\CC(q^{\frac12})$-algebra\footnote{The construction can be done on $\CC(q)$ but $q^{\frac12}$ allows to have more symmetry in the formulas.} with unit generated by $E_i$, $F_i$, $K_i$ and $K_i^{-1}$ for $i=1,2$ and subjected to the relations for $i$ and $j$ in $\{1,2\}$ and $i\neq j$:
\begin{align*}
&K_iK_i^{-1} = K_i^{-1} K_i=1,& \qquad &&K_1K_2 = K_2K_1,\\  
&K_iE_i=q^2E_iK_i,& \qquad &&K_iF_i = q^{-2}F_iK_i,\\
&K_iE_j=-q^{-1}E_jK_i,&\qquad &&K_iF_j = q F_jK_i,\\
&(q^1-q^{-1})(E_iF_i-F_iE_i)= K_i-K_i^{-1},&\qquad &&E_iF_j=F_jE_i,\\
&E_i^2E_j-[2]E_iE_jE_i + E_jE_i^2 = 0,&\qquad &&\\
&F_i^2F_j-[2]F_iF_jF_i + F_jF_i^2 = 0,&\qquad &&
\end{align*}
where $[n]$ stands for $\frac{q^n-q^{-n}}{q-q^{-1}}$. It can be given a structure of Hopf algebra (see \cite{LHRThese} for the formulas).
\end{dfn}
In the sequence we will consider two special $U_q(\sll_3)$-modules: $V^+$ spanned by $ e_{-1}^+$, $e_0^+$ and  $e_1^+$ and $V^-$ spanned by $e_{-1}^-$, $e_0^-$, $e_1^-$. The module $V^+$ is the $U_q(\sll_3)$-counterpart of the fundamental representation of $\sll_3$ and $V^-$ is its dual.
If $\epsilon= (\epsilon_1,\dots,\epsilon_l)$ is a finite sequence of signs, we denote by $V^\epsilon$ the $\CC(q^{\frac12})$-vector space
$\bigotimes_{i=1}^l V^{\epsilon_i}$ endowed with the structure of $U_q(\sll_3)$-module provided by the co-multiplication and by the action of $U_q(\sll_3)$ on $V^+$ and $V^-$. If $\epsilon$ is the empty sequence, then by convention $V^\epsilon$ is $\CC(q^{\frac12})$ with the structure of $U_q(\sll_3)$-module given by the co-unit.
We consider the following maps of $U_q(\sll_3)$-modules between the $V^\epsilon$'s:
\begin{align*}
  b^{+-}: \CC(q^{\frac12}) \to V^{+}\otimes V^{-}\qquad &1 \mapsto q e_{-1}^{+}\otimes e_1^{-} + e_0^{+}\otimes e_0^{-} + q^{-1} e_1^{+}\otimes e_{-1}^{-}, \\
  b^{-+}: \CC(q^{\frac12}) \to V^{-}\otimes V^{+}\qquad &1 \mapsto q^{-1} e_{-1}^{-}\otimes e_1^{+} + e_0^{-}\otimes e_0^{+} + q e_1^{-}\otimes e_{-1}^{+}, \\
  \sigma_{+-} : V^{+}\otimes V^{-} \to \CC(q^{\frac12}) \qquad & e_{-1}^{+}\otimes e_1^{-}\mapsto q,\,\quad e_0^{+}\otimes e_0^{-}\mapsto 1,\, \quad e_1^{+}\otimes e_{-1}^{-} \mapsto q^{-1},\\
  \sigma_{-+} : V^{-}\otimes V^{+} \to \CC(q^{\frac12}) \qquad & e_{-1}^{-}\otimes e_1^{+}\mapsto q^{-1},\,\quad e_0^{-}\otimes e_0^{+}\mapsto 1,\, \quad e_1^{-}\otimes e_{-1}^{+} \mapsto q,\\
 t^{+++}: \CC(q^{\frac12}) \to V^{(+,+,+)}\qquad &1 \mapsto q^{\frac{-3}{2}}e_{1}^+\otimes e_{0}^+ \otimes e_{-1}^+ + q^{\frac{-1}{2}}e_{0}^+\otimes e_{1}^+ \otimes e_{-1}^+ \\ &+ q^{\frac{-1}{2}}e_{1}^+\otimes e_{-1}^+ \otimes e_{0}^+ + q^{\frac12}e_{0}^+\otimes e_{-1}^+ \otimes e_{1}^+ \\ &+ q^{\frac12}e_{-1}^+\otimes e_{1}^+ \otimes e_{0}^+ + q^{\frac32}e_{-1}^+\otimes e_{0}^+ \otimes e_{1}^+,\\
 t^{---}: \CC(q^{\frac12}) \to V^{(-,-,-)}\qquad &1\mapsto q^{\frac{-3}{2}}e_{1}^-\otimes e_{0}^- \otimes e_{-1}^- + q^{\frac{-1}2}e_{0}^-\otimes e_{1}^- \otimes e_{-1}^- \\ &+ q^{\frac{-1}2}e_{1}^-\otimes e_{-1}^- \otimes e_{0}^- + q^{\frac12}e_{0}^-\otimes e_{-1}^- \otimes e_{1}^- \\ &+q^{\frac12}e_{-1}^-\otimes e_{1}^- \otimes e_{0}^- + q^{\frac32}e_{-1}^-\otimes e_{0}^- \otimes e_{1}^-.\\
\end{align*}
These maps are slightly different from the one defined in \cite{MR1684195}, but they are more symmetric in $q$ and $q^{-1}$.
\begin{req}\label{dfn:t}
  The maps $\sigma_{+-}$ and $\sigma_{-+}$ fix an isomorphism between the dual of $V^+$ (\resp the dual of $V^-$) and $V^-$ (\resp $V^+$). Under this isomorphisms the basis $(e_{-1}^+,e_{0}^+,e_{1}^+)$ and $(q^{-1}e_{1}^-,e_{0}^-,qe_{-1}^-)$ are dual to each other and the basis  $(e_{-1}^-,e_{0}^-,e_{1}^-)$ and $(qe_{1}^+,e_{0}^+,q^{-1}e_{-1}^-)$ are dual to each other.
\end{req}
Following Reshetikhin-Turaev~\cite{MR1036112}, we use a diagrammatic presentation of these maps (see figure~\ref{fig:graphica-maps}, they should be read from bottom to top): a vertical strand represents the identity, stacking diagrams one onto another corresponds to composition, and drawing two diagrams side by side corresponds to taking the tensor product of two maps.
\begin{figure}[!ht]
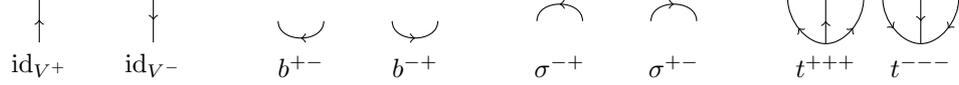

  \centering
 \begin{tikzpicture}[scale=0.6]
    \input{\imagesfolder/sw_graphicaldef0}
  \end{tikzpicture}
\hspace{0.8cm}
 \begin{tikzpicture}[scale=0.6]
    \input{\imagesfolder/sw_graphicaldef}
  \end{tikzpicture}
\hspace{0.8cm}
\begin{tikzpicture}[scale=0.6]
    \input{\imagesfolder/sw_graphicaldef2}
  \end{tikzpicture}
\hspace{0.8cm}
\begin{tikzpicture}[yscale=0.6, xscale = 0.5]
    \input{\imagesfolder/sw_graphicaldef3}
  \end{tikzpicture}
  \caption{Diagrammatic description of $b^{+-}$, $b^{-+}$, $\sigma_{-+}$, $\sigma_{+-}$, $t^{+++}$ and $t^{---}$.}
  \label{fig:graphica-maps}
\end{figure}
The $b$'s, $\sigma$'s and $t$'s yields some other maps (\eg $t_{-}^{++}\eqdef(\id_{V^{(+,+)}}\otimes \sigma_{+-})\circ(t^{+++}\otimes \id_{V^-}$) coherent with the diagrammatic presentation (see figure~\ref{fig:graphicadef2}). 
\begin{figure}[!ht]
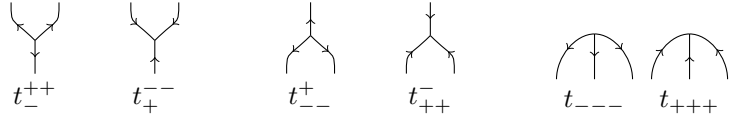

  \centering
  \begin{tikzpicture}[scale=0.7]
    \begin{scope}[scale=0.9]
    \input{\imagesfolder/sw_graphicaldef5}
  \end{scope}
\end{tikzpicture}
\hspace{1cm}
\begin{tikzpicture}[scale=0.7]
    \begin{scope}[scale=0.9]
    \input{\imagesfolder/sw_graphicaldef4}
  \end{scope}
\end{tikzpicture}
\hspace{1cm}
\begin{tikzpicture}[xscale=0.5, yscale = 0.6]
    \input{\imagesfolder/sw_graphicaldef6}
  \end{tikzpicture}
  \caption{The maps $t_{-}^{++}$, $t_{+}^{--}$, $t_{--}^{+}$, $t_{++}^{-}$, $t_{---}$ and $t_{+++}$.}
  \label{fig:graphicadef2}
\end{figure}
 \begin{req}\label{req:graphical-is-fine}
It is a straightforward computation to check that the $b$'s, the $\sigma$'s and the $t$'s can be composed so that the diagrammatic representation makes sense. We mean that two isotopic diagrams represent the same map. One only need to check that the wave moves hold: 
\[(\id_{V^+}\otimes \sigma_{-+})\circ (b^{+-}\otimes \id_{V^+})= \id_{V^+},\]and similar equalities with signs and orders changed. It can be computed by hands. 
 \end{req}
\begin{prop}[Kuperberg, \cite{MR1403861}]\label{prop:skein-relation}
  We have the following equalities between maps of $U_q(\sll_3)$-modules:
  \begin{align*}
   \kup{\!\websquare[0.4]} &= \kup{\!\webtwovert[0.4]} + \kup{\!\webtwohori[0.4]}, \\
   \kup{\!\webbigon[0.4]}  &= [2] \cdot \kup{\!\webvert[0.4]},\\
   \kup{\!\webcircle[0.4]} &= \kup{\!\webcirclereverse[0.4]} = [3]
  \end{align*}
\end{prop}
 \subsubsection{The Kuperberg bracket}
 \label{sec:kuperberg-bracket}

  In what follows, we develop the diagrammatic point of view and use the representation theory to define the Kuperberg bracket.
\begin{dfn}[Kuperberg, \cite{MR1403861}]\label{dfn:closedweb} 
  A \emph{closed web} is a cubic oriented graph (with possibly some vertex-less loops) smoothly embedded in $\RR^2$ such that every vertex is either a sink or a source.
\end{dfn}
\begin{figure}[!ht]
  \centering
  \begin{tikzpicture}[yscale= 0.6, xscale= 0.6]
    \begin{scope}
   [yscale = {1}, xscale={1},decoration={markings, mark=at
     position 0.5 with {\arrow{>}}},postaction={decorate}]
\coordinate (A) at (0,0);
\coordinate (B) at (1,0);
\coordinate (C) at (2,-0.5);
\coordinate (D) at (3,0);
\coordinate (E) at (4,0);
\coordinate (F) at (5,0.5);
\coordinate (A1) at (0,1);
\coordinate (B1) at (1,1);
\coordinate (C1) at (2,1.5);
\coordinate (D1) at (3,1);
\coordinate (E1) at (4,1);
\coordinate (F1) at (4,2);
\coordinate (G) at (5,-0.5);
\coordinate (H) at (3.5,-1);

\draw[postaction=decorate] (A1) --(A);
\draw[postaction=decorate] (A1) .. controls +(-0.5,0)  and  +(-0.5,0).. (A);
\draw[postaction=decorate] (A1) -- (B1);
\draw[postaction=decorate] (B)-- (A);
\draw[postaction=decorate] (B)--(B1);
\draw[postaction=decorate] (B)--(C);
\draw[postaction=decorate] (C1)--(B1);
\draw[postaction=decorate] (C1)--(D1);
\draw[postaction=decorate] (C1)--(F1);
\draw[postaction=decorate] (D)--(C);
\draw[postaction=decorate] (D)--(D1);
\draw[postaction=decorate] (D)--(E);
\draw[postaction=decorate] (F) -- (G);
\draw[postaction=decorate] (F)-- (E);
\draw[postaction=decorate] (F) .. controls +(0,0.5) and  +(0.4,0.4) .. (F1);
\draw[postaction=decorate] (E1)--(F1);
\draw[postaction=decorate] (E1)--(D1);
\draw[postaction=decorate] (E1)--(E);
\draw[postaction=decorate] (H)--(C);
\draw[postaction=decorate] (H) .. controls +(0.5,0) and  +(0,-0.5) .. (G);
\draw[postaction=decorate] (H) .. controls +(0.2,0.3) and  +(-0.4,-0.1) .. (G);
\draw[postaction=decorate] (7,1) circle (0.5cm);
\end{scope}
  \end{tikzpicture}  
  \caption{Example of a closed web.}
  \label{fig:example-closed-web}
\end{figure}
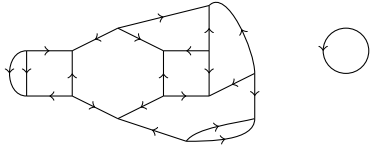
\begin{req}
  By graph we don't mean simple graph, so that a web may have multi-edges. The orientation condition is equivalent to say that the graph is bipartite (by sinks and sources). The vertex-less loops may be a strange notion from the graph theoretic point of view, to prevent this we could have introduced some meaningless 2-valent vertices\footnote{In this case the orientation condition is: around each vertex the flow module 3 is preserved.}.
\end{req}

Thanks to section~\ref{sec:u_qsll_3} a web can be interpreted as a map from $\CC(q^{\frac12})$ to $\CC(q^{\frac12})$, and thanks to the following proposition we can compute it fully combinatorially. 
\begin{prop}\label{prop:closed2elliptic}
  Every closed web contains at least, a circle, a digon or a square.
\end{prop}
The proof is easy and remains on a Euler characteristic argument. One can find it in \cite{MR1172374}.

\begin{dfn}\label{dfn:webtangle}
  A \emph{$(\epsilon_0,\epsilon_1)$-web tangle} $w$ is an intersection of a closed web $w'$ with $\RR\times [0,1]$ such that:
  \begin{itemize}
  \item there exists $\eta_0 \in ]0,1]$ such that $w\cap \RR \times [0,\eta_0] = \{1,2,\dots,l(\epsilon_0)\}\times [0,\eta_0]$,
  \item there exists $\eta_1 \in [0,1[$ such that $w\cap \RR \times [\eta_1,1] = \{1,2,\dots,l(\epsilon_1)\}\times [\eta_1,1]$,
  \item the orientations of the edges of $w$, match $-\epsilon_0$ and $\epsilon_1$ (see figure \ref{fig:exampl_webtangle} to have the conventions).
  \end{itemize}
An \emph{$\epsilon$-web} is a $(\epsilon,\emptyset)$-web tangle. If $w$ is an $\epsilon$-web, we define $\partial w \eqdef\epsilon$. The \emph{length of $\epsilon$} is its length as a finite sequence, it is denoted by $l(\epsilon)$.

If $w$ is a $(\epsilon_0,\epsilon_1)$-web tangle, we say that $\bar{w}$ is \emph{the conjugate of $w$} if it is the $(\epsilon_1,\epsilon_0)$-web tangle obtained from $w$ by taking the symmetric of $w$ with respect to the line $\RR\times\{\frac12\}$ and by changing all the orientations (see figure \ref{fig:exampl_webtangle}). It is clear that $\bar{\bar{w}}=w$.
\end{dfn}
\begin{figure}[!ht]
  \centering
  \begin{tikzpicture}[yscale= 0.25, xscale= 0.45]
    \input{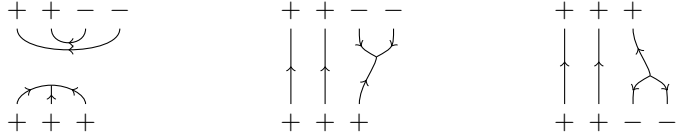}
  \end{tikzpicture}
  \caption{Two $(\epsilon_0,\epsilon_1)$-web tangles  and one $(\epsilon_1,\epsilon_0)$-web tangle with $\epsilon_0 = (+,+,+)$ and $\epsilon_1=(+,+,-,-)$. The two last ones are conjugate.}
  \label{fig:exampl_webtangle}
\end{figure}
\begin{dfn}
  An $(\epsilon_0,\epsilon_1)$-web tangle or an $\epsilon$-web is \emph{non-elliptic} if it contains no vertex-less loop, no digon and no square.
\end{dfn}

One may wonder if, given a sequence of signs $\epsilon$ there exists or not an $\epsilon$-web. The answer is actually quite easy, there exists if and only if the sum of the signs is a multiple of three. Such a sequence of signs is \emph{admissible}. In the sequel, we will only consider admissible sequences of signs.

\begin{thm}[Kuperberg,\cite{MR1403861}]\label{thm:Kuperberg}
  Let $\epsilon$ be an admissible sequences of signs, then the set $(\kup{w})_{w\in \NE(\epsilon)}$ is a base of $\hom_{U_q(\sll_3)}(V^{\epsilon}, \CC(q^{\frac12}) )$ where $\NE(\epsilon)$ is a set of representatives of isotopy classes of non-elliptic $\epsilon$-webs. This shows in particular that $\NE(\epsilon)$ is finite. 
\end{thm}
\begin{notation}\label{not:Wepsilon}
  Let us denote by $W^\epsilon$ the $\ZZ[q,q^{-1}]$-modules generated by $\epsilon$-webs. The previous theorem can be rephrases like this:
\[
\hom_{U_q(\sll_3)}(V^{\epsilon}, \CC(q^{\frac12}) ) \simeq W^\epsilon \otimes_{\ZZ[q,q^{-1}]}] \CC(q^{\frac12})
\]
\end{notation}

\subsubsection{Just enough about colorings}
\label{sec:just-enough-about}

\begin{dfn}
  Let $w$ be a web-tangle. A \emph{coloring} of $w$ is an application from the set of edges of $w$ to $\{-1,0,1\}$ such that at each vertex, all the colors are different (or equivalently all present). If $w$ is a web-tangle and $c$ is a coloring o $w$, the symbol $w_c$ means \emph{the web-tangle $w$ colored by $c$} \ie the data given by $w$ and $c$.

\end{dfn}
\begin{prop}[\cite{MR1172374}]
  Let $w$ be a closed web, the Kuperberg bracket evaluated in 1 is equal to the number of colorings of $w$.
\end{prop}
\begin{req}
The relation between colorings and the representation theorety of $\sll_3$ is quite clear: to the color $i$ one may affect the element $e^\pm_{\pm i}$ of $V^\pm$. The fact that $t_{+++}$  represents an isomorphism of $(V^+)^{\wedge 3}$ correspond to the condition on colors at vertices of webs. Of course, with this analogy we lost the $q$ deformation, but it can be recover (see \cite[chapter 5]{LHRThese}).
\end{req}

\begin{dfn}
  If $\epsilon$ is a sequence of signs, then a \emph{coloring} of $\epsilon$ is simply a function from $\epsilon$ to  $\{-1,0,1\}$. If $w_c$ is colored $\epsilon$-web, the coloring $c$ induces a coloring $c'$ of $\epsilon$. We say that $c$ \emph{restricts} to $c'$ on $\epsilon$.
\end{dfn}
\begin{notation}
  We denote by $\col(\epsilon)$ (\resp $\col(w)$) the set of colorings of a sequence of signs $\epsilon$ (\resp of a web (or web-tangle) $w$) and if $c$ is a coloring of $\epsilon$ and $w$ an $\epsilon$-web, we denote by $\col_c(w)$ the set of colorings of $w$ which restrict to $c$ on $\epsilon$.
\end{notation}
The following lemma can be easily derived from \cite[theorem 2]{MR1684195}:
\begin{lem}\label{lem:uniquecoloring}
  Let $(w_i)_{i \in I}$ be a finite collection of non-elliptic $\epsilon$-webs, then there exists a coloring $c$ of $\epsilon$ such that there exists a unique $i_0$ in $I$ and a unique coloring $c_{i_0}$ of $w_{i_0}$ such that $c_{i_0}$ restricts to $c$ on $\epsilon$.  
\end{lem}
\begin{proof}[Sketch of the from {\cite[theorem 2]{MR1684195}}.]
  One just need to pick up a maximal non-elliptic web (ordered via the so-called state string using the lexicographical order, see \cite{MR1684195}).
\end{proof}
\subsection{Khovanov's TQFT}
\label{sec:khovanovs-tqft}

\subsubsection{Foams} All material here comes from \cite{MR2100691}.
\label{sec:foams}
\begin{dfn}
  A \emph{pre-foam} is a smooth oriented compact surface $\Sigma$ (its connected component are called \emph{facets}) together with the following data~:
\begin{itemize}
\item A partition of the connected components of the boundary into cyclically ordered 3-sets and for each 3-set $(C_1,C_2,C_3)$, three orientation-preserving diffeomorphisms $\phi_1:C_2\to C_3$, $\phi_2:C_3\to C_1$ and $\phi_3:C_1\to C_2$ such that $\phi_3 \circ \phi_2 \circ \phi_1 = \mathrm{id}_{C_2}$.
\item A function from the set of facets to the set of non-negative integers (this gives the number of \emph{dots} on each facet).
\end{itemize}
The \emph{CW-complex associated with a pre-foam} is the 2-dimensional CW-complex $\Sigma$ quotiented by the diffeomorphisms so that the three circles of one 3-set are identified and become just one called a \emph{singular circle}.
The \emph{degree} of a pre-foam $f$ is equal to $-2\chi(\Sigma')$ where $\chi$ is the Euler characteristic and $\Sigma'$ is the CW-complex associated with $f$ with the dots punctured out (\ie a dot increases the degree by 2).
\end{dfn}
\begin{req}
  The CW-complex has two local models depending on whether we are on a singular circle or not. If a point $x$ is not on a singular circle, then it has a neighborhood diffeomorphic to a 2-dimensional disk, else it has a neighborhood diffeomorphic to a Y shape times an interval (see figure \ref{fig:yshape}).
  \begin{figure}[h]
    \centering
    \begin{tikzpicture}[scale=0.8]
      \input{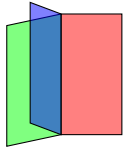}
    \end{tikzpicture}
    \caption{Singularities of a pre-foam}
    \label{fig:yshape}
  \end{figure}
\end{req}
\begin{dfn}
  A \emph{closed foam} is the image of an embedding of the CW-complex associated with a pre-foam such that the cyclic orders of the pre-foam are compatible with the left-hand rule in $\RR^3$ with respect to the orientations of the singular circles\footnote{We mean here that if, next to a singular circle, with the forefinger of the left hand we go from face 1 to face 2 to face 3 the thumb points to indicate the orientation of the singular circle (induced by orientations of facets). This is not quite canonical, physicists use more the right-hand rule, however this is the convention used in \cite{MR2100691}.}. The \emph{degree} of a closed foam is the degree of the underlying pre-foam. 
\end{dfn}
\begin{dfn}\label{dfn:wwfoam}
  If $w_b$ and $w_t$ are $(\epsilon_0,\epsilon_1)$-web-tangles, a \emph{$(w_b,w_t)$-foam} $f$ is the intersection of a foam $f'$ with $\RR\times [0,1]\times[0,1]$ such that
  \begin{itemize}
  \item there exists $\eta_0 \in ]0,1]$ such that $f\cap \RR \times [0,\eta_0]\times [0,1] = \{\frac{1}{2l(\epsilon_0)}, \frac{1}{2l(\epsilon_0)} + \frac{1}{l(\epsilon_0)},  \frac{1}{2l(\epsilon_0)} + \frac{2}{l(\epsilon_0)}, \dots, \frac{1}{2l(\epsilon_0)} + \frac{l(\epsilon_0)-1}{l(\epsilon_0)} \}\times [0,\eta_0]\times [0,1]$,
  \item there exists $\eta_1 \in [0,1[$ such that $f\cap \RR \times [\eta_1,1]\times [0,1] = \{\frac{1}{2l(\epsilon_1)}, \frac{1}{2l(\epsilon_1)} + \frac{1}{l(\epsilon_1)},  \frac{1}{2l(\epsilon_1)} + \frac{2}{l(\epsilon_1)}, \dots, \frac{1}{2l(\epsilon_1)} + \frac{l(\epsilon_1)-1}{l(\epsilon_1)}\}\times [\eta_1,1]\times [0,1]$,
  \item there exists $\eta_b \in ]0,1]$ such that $f\cap \RR \times [0,1 ]\times  [0, \eta_b] = w_b \times [0, \eta_b]$,
  \item there exists $\eta_t \in [0,1[$ such that $f\cap \RR \times [0,1 ]\times  [\eta_t, 1] = w_t \times [\eta_t,1]$,
  \end{itemize}
with compatibility of orientations of the facets of $f$ with the orientation of $w_t$ and the reversed orientation of $w_b$.
The \emph{degree} of a $(w_b,w_t)$-foam $f$ is equal to $\chi(w_b)+\chi(w_t)-2\chi(\Sigma)$ where $\Sigma$ is the underlying CW-complex associated with $f$ with the dots punctured out. 
\end{dfn}
If $f_b$ is a $(w_b,w_m)$-foam and $f_t$ is a $(w_m, w_t)$-foam we define $f_bf_t$ to be the $(w_b,w_t)$-foam obtained by gluing $f_b$ and $f_t$ along $w_m$ and resizing. This operation may be thought as a composition if we think of a $(w_1,w_2)$-foam as a morphism from $w_2$ to $w_1$ \ie from the top to the bottom. This composition map is degree preserving. Like for the webs, we define the \emph{mirror image} of a $(w_1,w_2)$-foam $f$ to be the $(w_2,w_1)$-foam which is the mirror image of $f$ with respect to $\RR\times\RR\times \{\frac12\}$ with all orientations reversed. We denote it by $\bar{f}$.
\begin{dfn}
  If $\epsilon_0=\epsilon_1=\emptyset$ and $w$ is a closed web, then a $(\emptyset,w)$-foam is simply called \emph{foam} or \emph{$w$-foam} when one wants to focus on the boundary of the foam.
\end{dfn}
All these data together lead  to the definition of a monoidal 2-category. 
\begin{dfn}
  The 2-category $\mathcal{WT}$ is the monoidal\footnote{Here we choose a rather strict point of view and hence the monoidal structure is strict (we consider everything up to isotopy), but it is possible to define the notion in a non-strict context, and the same data gives us a monoidal bicategory.} 2-category given by the following data~:
  \begin{itemize}
  \item The objects are finite sequences of signs,
  \item The 1-morphisms from $\epsilon_1$ to $\epsilon_0$ are isotopy classes (with fixed boundary) of $(\epsilon_0,\epsilon_1)$-web-tangles,
  \item The 2-morphisms from $\widehat{w_t}$ to $\widehat{w_b}$ are $\QQ$-linear combinations of isotopy classes of $(w_b,w_t)$-foams, where\ \ $\widehat{\cdot}$\ \ stands for the ``isotopy class of''. The 2-morphisms come with a grading, the composition respects the degree.
  \end{itemize}
The monoidal structure is given by concatenation of sequences at the $0$-level, and disjoint union of vertical strands or disks (with corners) at the $1$ and $2$ levels. 
\end{dfn}
\subsubsection{Khovanov's TQFT for web-tangles}
\label{sec:khovanov-tqft-web}
In \cite{MR2100691}, Khovanov defines a numerical invariant for pre-foams and this allows him to construct a TQFT $\mathcal{F}$ from the category $\hom_{\mathcal{WT}}(\emptyset,\emptyset)$ to the category of graded $\QQ$-modules, (via a universal construction à la BHMV 
\cite{MR1362791}). This TQFT is graded (this comes from the fact that pre-foams with non-zero degree are evaluated to zero), and satisfies the following local relations (brackets indicate grading shifts)~:
\begin{align*}
\mathcal{F}\left(\websquare[0.4]\,\right) &= 
\mathcal{F}\left({\webtwovert[0.4]}\,\right) \oplus
\mathcal{F}\left({\webtwohori[0.4]}\,\right), \\
\mathcal{F}\left({\webbigon[0.4]}\,\right)  &= 
\mathcal{F}\left({\webvert[0.4]}\,\right)\{-1\}\oplus \mathcal{F}\left({\webvert[0.4]}\right)\{1\},\\
\mathcal{F}\left({\webcircle[0.4]\,}\right) &= 
\mathcal{F}\left({\webcirclereverse[0.4]}\,\right) = \QQ\{-2\} \oplus \QQ \oplus\QQ\{2\}.
  \end{align*}
These relations show that $\mathcal{F}$ is a categorified counterpart of the Kuperberg bracket. We sketch the construction below.
\begin{dfn}
We denote by $\mathcal{A}$ the Frobenius algebra $\ZZ[X]/(X^3)$ with trace $\tau$ given by:
\[\tau(X^2)=-1, \quad \tau(X)=0, \quad \tau(1)=0.\] 
We equip $\mathcal{A}$ with a graduation by setting $\deg(1)=-2$, $\deg(X)=0$ and $\deg(X^2)=2$. With these settings, the multiplication has degree 2 and the trace has degree -2. The co-multiplication is determined by the multiplication and the trace and we have:
\begin{align*}
  &\Delta(1) = -1\otimes X^2 - X\otimes X - X^2\otimes 1 \\
  &\Delta(X) = -X\otimes X^2 - X^2\otimes X \\
  &\Delta(X^2) = -X^2\otimes X^2
\end{align*}
\end{dfn}
This Frobenius algebra gives us a 1+1 TQFT (this is well-known, see~\cite{MR2037238} for details), we denote it by $\mathcal{F}$: the circle is sent on $\mathcal{A}$, a cup on the unity, a cap on the trace, and a pair of pants either on the multiplication or the co-multiplication. A dot on a surface represents multiplication by $X$ so that $\mathcal{F}$ extends to the category of oriented dotted (1+1)-cobordisms.
We have a surgery formula given by figure~\ref{fig:surg} and an evaluation for dotted sphere: a sphere with two dots evaluates to $-1$, other dotted spheres evaluates to $0$. 
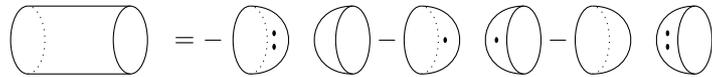
\begin{figure}[h!]
  \centering
  \begin{tikzpicture}[scale=0.45]
    \begin{scope}[xshift=-0.5cm]
       \draw (0,-1) arc (270:90:0.5 and 1);
      \draw[dotted] (0,-1) arc (-90:90:0.5 and 1);
      \draw (3,0) ellipse (0.5 and 1);
      \draw (0,-1) -- (3,-1);
      \draw (0,1) -- (3,1);
\end{scope}
\node at (4.5,0) {$=-$};
\node at (10,0) {$-$};
\node at (15,0) {$-$};
\begin{scope}[xshift=6cm]
       \draw (0,-1) arc (270:90:0.5 and 1);
      \draw[dotted] (0,-1) arc (-90:90:0.5 and 1);
      \draw (3,0) ellipse (0.5 and 1);
      \draw (0,-1) .. controls +(1.5,0) and +(1.5,0) .. (0,1);
      \draw (3,-1) .. controls +(-1.5,0) and +(-1.5,0) .. (3,1);
      \filldraw (0.7,0.2) ellipse (1pt and 2pt);
      \filldraw (0.7,-0.2) ellipse (1pt and 2pt);
\end{scope}
\begin{scope}[xshift=11cm]
       \draw (0,-1) arc (270:90:0.5 and 1);
      \draw[dotted] (0,-1) arc (-90:90:0.5 and 1);
      \draw (3,0) ellipse (0.5 and 1);
      \draw (0,-1) .. controls +(1.5,0) and +(1.5,0) .. (0,1);
      \draw (3,-1) .. controls +(-1.5,0) and +(-1.5,0) .. (3,1);
      \filldraw (0.7,0) ellipse (1pt and 2pt);
      \filldraw (2.2,0) ellipse (1pt and 2pt);
\end{scope}
\begin{scope}[xshift=16cm]
       \draw (0,-1) arc (270:90:0.5 and 1);
      \draw[dotted] (0,-1) arc (-90:90:0.5 and 1);
      \draw (3,0) ellipse (0.5 and 1);
      \draw (0,-1) .. controls +(1.5,0) and +(1.5,0) .. (0,1);
      \draw (3,-1) .. controls +(-1.5,0) and +(-1.5,0) .. (3,1);
      \filldraw (2.2,-0.2) ellipse (1pt and 2pt);
      \filldraw (2.2,0.2) ellipse (1pt and 2pt);
\end{scope}
  \end{tikzpicture}
  \caption{The surgery formula for the TQFT $\mathcal{F}$.}
  \label{fig:surg}
\end{figure}

 This TQFT gives of course a numerical invariant for closed dotted oriented surfaces. If one defines numerical values for the differently dotted theta pre-foams (the theta pre-foam consists of 3 disks with trivial diffeomorphisms between their boundary see figure \ref{fig:thetapre}) then by applying the surgery formula and using the values of dotted spheres given by the TQFT, one is able to compute a numerical value for all pre-foams. 

\begin{figure}[h!]
  \centering
  \begin{tikzpicture}[scale = 0.65]
    \input{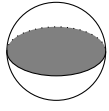}
  \end{tikzpicture}
  \caption{The dotless theta pre-foam.}
  \label{fig:thetapre}
\end{figure}
\begin{figure}[h!]
  \centering
\begin{tikzpicture}[scale = 0.65]
\input{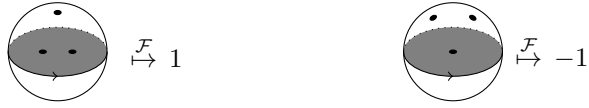}
\end{tikzpicture}
  \caption{The evaluations of dotted theta foams, the evaluation is unchanged when one cyclically permutes the faces. All the configurations which cannot be obtained from these by cyclic permutation are sent to $0$ by $\mathcal{F}$.}
  \label{fig:thetaeval}
\end{figure}

In \cite{MR2100691}, Khovanov shows that setting the evaluations of the dotted theta foams as shown on figure \ref{fig:thetaeval}, leads to a well defined numerical invariant $\mathcal{F}$ for pre-foams. This numerical invariant gives the opportunity to build a (closed web, $(\cdot,\cdot)$-foams)-TQFT: for a given web $w$, consider the $\QQ$-module generated by all the $(w,\emptyset)$-foams, and mod this space out by the kernel of the bilinear map $(f,g)\mapsto \mathcal{F}(\bar{f}g)$. Note that $\bar{f}g$ is a closed foam. Khovanov showed that the obtained graded vector spaces is finite dimensional with graded dimensions given by the Kuperberg formulas, and he showed that we have the local relations described on figure~\ref{fig:localrel}.

\begin{figure}[h]
  \centering 
  \begin{tikzpicture}[scale=0.25]
\begin{scope}[xshift=0cm, yshift= 0cm, decoration={markings, mark=at
     position 0.5 with {\arrow{>}}},postaction={decorate}]
\draw[postaction= {decorate}] (-2,0) arc (-180:0:2 and 1);
\draw (-2,0) arc (180:0:2);
\draw[dashed] (-2,0) arc (180:0:2 and 1);
\draw[dashed] (-2,0) arc (-180:-135:2);
\draw[dashed] (2,0) arc (0:-45:2);
\draw (0,0)++(-135:2) arc (-135:-45:2);
\draw[dashed] (0,0) +(135:2) -- +(45:2);
\draw (0,0)++(135:2) --  ++(-0.7,0) -- ++(-135:4) -- ++(7,0) -- ++(45:4) -- (45:2);
\filldraw (0,1.8) ellipse (3pt and 1.5pt);
\end{scope}
\node at (5,0) {$=-$};
\begin{scope}[xshift=10cm, yshift= 0cm, decoration={markings, mark=at
     position 0.5 with {\arrow{>}}},postaction={decorate}]
\draw[postaction= {decorate}] (-2,0) arc (-180:0:2 and 1);
\draw (-2,0) arc (180:0:2);
\draw[dashed] (-2,0) arc (180:0:2 and 1);
\draw[dashed] (-2,0) arc (-180:-135:2);
\draw[dashed] (2,0) arc (0:-45:2);
\draw (0,0)++(-135:2) arc (-135:-45:2);
\draw[dashed] (0,0) +(135:2) -- +(45:2);
\draw (0,0)++(135:2) --  ++(-0.7,0) -- ++(-135:4) -- ++(7,0) -- ++(45:4) -- (45:2);
\filldraw (0,-1.8) ellipse (3pt and 1.5pt);
\end{scope}
\node at (15,0) {$=$};
\begin{scope}[xshift=20cm, yshift= 0cm, decoration={markings, mark=at
     position 0.5 with {\arrow{>}}},postaction={decorate}]
\draw (0,0)++(135:2) --  ++(-0.7,0) -- ++(-135:4) -- ++(7,0) -- ++(45:4) -- (135:2);
\end{scope}
\begin{scope} [yshift= -6cm]
  \begin{scope}[xshift=0cm, yshift= 0cm, decoration={markings, mark=at
     position 0.5 with {\arrow{>}}},postaction={decorate}]
\draw[postaction= {decorate}] (-2,0) arc (-180:0:2 and 1);
\draw (-2,0) arc (180:0:2);
\draw[dashed] (-2,0) arc (180:0:2 and 1);
\draw[dashed] (-2,0) arc (-180:-135:2);
\draw[dashed] (2,0) arc (0:-45:2);
\draw (0,0)++(-135:2) arc (-135:-45:2);
\draw[dashed] (0,0) +(135:2) -- +(45:2);
\draw (0,0)++(135:2) --  ++(-0.7,0) -- ++(-135:4) -- ++(7,0) -- ++(45:4) -- (45:2);
\end{scope}
\node at (5,0) {$=$};
\begin{scope}[xshift=10cm, yshift= 0cm, decoration={markings, mark=at
     position 0.5 with {\arrow{>}}},postaction={decorate}]
\draw[postaction= {decorate}] (-2,0) arc (-180:0:2 and 1);
\draw (-2,0) arc (180:0:2);
\draw[dashed] (-2,0) arc (180:0:2 and 1);
\draw[dashed] (-2,0) arc (-180:-135:2);
\draw[dashed] (2,0) arc (0:-45:2);
\draw (0,0)++(-135:2) arc (-135:-45:2);
\draw[dashed] (0,0) +(135:2) -- +(45:2);
\draw (0,0)++(135:2) --  ++(-0.7,0) -- ++(-135:4) -- ++(7,0) -- ++(45:4) -- (45:2);
\filldraw (0,1.8) ellipse (3pt and 1.5pt);
\filldraw (0,-1.8) ellipse (3pt and 1.5pt);
\end{scope}
\node at (15,0) {$=$};
\begin{scope}[xshift=20cm, yshift= 0cm, decoration={markings, mark=at
     position 0.5 with {\arrow{>}}},postaction={decorate}]
\draw (0,0)++(135:2) --  ++(-0.7,0) -- ++(-135:4) -- ++(7,0) -- ++(45:4) -- (135:2);
\filldraw (0.5,0) ellipse (3pt and 1.5pt);
\filldraw (0,0) ellipse (3pt and 1.5pt);
\filldraw (-0.5,0) ellipse (3pt and 1.5pt);
\end{scope}
\node at (25,0) {$=0$};
\end{scope}
\begin{scope}[yshift=-12cm]
  \begin{scope}[xshift=0cm, yshift= 0cm, decoration={markings, mark=at
     position 0.5 with {\arrow{>}}},postaction={decorate}]
\draw[postaction= {decorate}] (-2,0) arc (-180:0:2 and 1);
\draw (-2,0) arc (180:0:2);
\draw[dashed] (-2,0) arc (180:0:2 and 1);
\draw[dashed] (-2,0) arc (-180:-135:2);
\draw[dashed] (2,0) arc (0:-45:2);
\draw (0,0)++(-135:2) arc (-135:-45:2);
\draw[dashed] (0,0) +(135:2) -- +(45:2);
\draw (0,0)++(135:2) --  ++(-0.7,0) -- ++(-135:4) -- ++(7,0) -- ++(45:4) -- (45:2);
\filldraw (-0.2,-1.8) ellipse (3pt and 1.5pt);
\filldraw (0.2,-1.8) ellipse (3pt and 1.5pt);
\end{scope}
\node at (5,0) {$=-$};
\begin{scope}[xshift=10cm, yshift= 0cm, decoration={markings, mark=at
     position 0.5 with {\arrow{>}}},postaction={decorate}]
\draw[postaction= {decorate}] (-2,0) arc (-180:0:2 and 1);
\draw (-2,0) arc (180:0:2);
\draw[dashed] (-2,0) arc (180:0:2 and 1);
\draw[dashed] (-2,0) arc (-180:-135:2);
\draw[dashed] (2,0) arc (0:-45:2);
\draw (0,0)++(-135:2) arc (-135:-45:2);
\draw[dashed] (0,0) +(135:2) -- +(45:2);
\draw (0,0)++(135:2) --  ++(-0.7,0) -- ++(-135:4) -- ++(7,0) -- ++(45:4) -- (45:2);
\filldraw (0.2,1.8) ellipse (3pt and 1.5pt);
\filldraw (-0.2,1.8) ellipse (3pt and 1.5pt);
\end{scope}
\node at (15,0) {$=$};
\begin{scope}[xshift=20cm, yshift= 0cm, decoration={markings, mark=at
     position 0.5 with {\arrow{>}}},postaction={decorate}]
\draw (0,0)++(135:2) --  ++(-0.7,0) -- ++(-135:4) -- ++(7,0) -- ++(45:4) -- (135:2);
\filldraw (0,0) ellipse (3pt and 1.5pt);
\end{scope}
\end{scope}
\end{tikzpicture}

\vspace{0.7cm}
\begin{tikzpicture}[scale=0.24]

\begin{scope}[xshift=0cm, yshift= 0cm, decoration={markings, mark=at
     position 0.5 with {\arrow{>}}},postaction={decorate}]
\draw (0,-2) arc (270:90:1 and 2);
\draw[dashed] (0,-2) arc (-90:90:1 and 2);
\fill[gray] (3,0) ellipse (1 and 2);
\draw (3,-2) arc (270:90:1 and 2);
\draw[dashed] (3,-2) arc (-90:90:1 and 2);
\fill[gray] (6,0) ellipse (1 and 2);
\draw (6,-2) arc (270:90:1 and 2);
\draw[dashed] (6,-2) arc (-90:90:1 and 2);
\draw (9,0) ellipse (1 and 2);
\draw (0,-2) -- (9,-2);
\draw (0,2) -- (9,2);
\node at (12,0) {$=-$};
\draw (15,-2) arc (270:90:1 and 2);
\draw[dashed] (15,-2) arc (-90:90:1 and 2);
\draw (15,-2) arc (-90:90:2 and 2);
\draw (20,-2) arc (-90:-270:2 and 2);
\draw(20,0) ellipse (1 and 2);
\end{scope}

\begin{scope}[xshift=-2cm, yshift= -6cm, decoration={markings, mark=at
     position 0.5 with {\arrow{>}}},postaction={decorate}]
\draw (0,-2) arc (270:90:1 and 2);
\draw[dashed] (0,-2) arc (-90:90:1 and 2);
\fill[gray] (3,0) ellipse (1 and 2);
\draw[postaction= {decorate}] (3,-2) arc (270:90:1 and 2);
\draw[dashed] (3,-2) arc (-90:90:1 and 2);
\draw (6,0) ellipse (1 and 2);
\draw (0,-2) -- (6,-2);
\draw (0,2) -- (6,2);
\node at (8,0) {$=$};
\draw (10,-2) arc (270:90:1 and 2);
\draw[dashed] (10,-2) arc (-90:90:1 and 2);
\draw (10,-2) arc (-90:90:2 and 2);
\draw (15,-2) arc (-90:-270:2 and 2);
\draw(15,0) ellipse (1 and 2);
\fill (15,0) ellipse (2pt and 4pt);
\node at (17,0) {$-$};
\fill (19,0) ellipse (2pt and 4pt);b
\draw (19,-2) arc (270:90:1 and 2);
\draw[dashed] (19,-2) arc (-90:90:1 and 2);
\draw (19,-2) arc (-90:90:2 and 2);
\draw (24,-2) arc (-90:-270:2 and 2);
\draw(24,0) ellipse (1 and 2);
\end{scope} 
\end{tikzpicture}

\vspace{0.7cm}
\begin{tikzpicture}[scale=0.4]
\begin{scope}[xshift=0cm, yshift= 0cm, decoration={markings, mark=at
     position 0.5 with {\arrow{>}}},postaction={decorate}]
\draw (-2,0) -- +(1,0);
\draw (1,0) -- +(1,0);
\draw[dashed] (-1,0) .. controls +(1,0.5) and +(-1,0.5).. +(2,0);
\draw (-1,0) .. controls +(1,-0.5) and +(-1,-0.5).. +(2,0);
\draw (-2,3) -- +(1,0);
\draw (1,3) -- +(1,0);
\draw (-1,3) .. controls +(1,0.5) and +(-1,0.5).. +(2,0);
\draw (-1,3) .. controls +(1,-0.5) and +(-1,-0.5).. +(2,0);
\draw (-2,0) -- +(0,3);
\draw[postaction={decorate}] (-1,0) -- +(0,3);
\draw[postaction={decorate}] (1,3) -- +(0,-3);
\draw (2,0) -- +(0,3);
\end{scope}
\node at (3,1.5) {$=$};
\begin{scope}[xshift=6cm, yshift= 0cm, decoration={markings, mark=at
     position 0.5 with {\arrow{>}}},postaction={decorate}]
\draw (-2,0) -- +(1,0);
\draw (1,0) -- +(1,0);
\draw[dashed] (-1,0) ..controls +(1,0.5) and +(-1,0.5).. +(2,0);
\draw (-1,0) ..controls +(1,-0.5) and +(-1,-0.5).. +(2,0);
\draw (-2,3) -- +(1,0);
\draw (1,3) -- +(1,0);
\draw (-1,3) ..controls +(1,0.5) and +(-1,0.5).. +(2,0);
\draw (-1,3) ..controls +(1,-0.5) and +(-1,-0.5).. +(2,0);
\draw (-2,0) -- +(0,3);
\draw[postaction={decorate}] (-1,0) .. controls +(0,1.5) and +(0,1.5).. +(2,0);
\draw[postaction={decorate}] (1,3) .. controls +(0,-1.5) and +(0,-1.5).. +(-2,0);
\draw (2,0) -- +(0,3);
\fill (0,3) circle (2pt and 2pt);
\end{scope}
\node at (9,1.5) {$-$}; 
\begin{scope}[xshift=12cm, yshift= 0cm, decoration={markings, mark=at
     position 0.5 with {\arrow{>}}},postaction={decorate}]
\draw (-2,0) -- +(1,0);
\draw (1,0) -- +(1,0);
\draw[dashed] (-1,0).. controls +(1,0.5) and +(-1,0.5).. +(2,0);
\draw (-1,0) ..controls +(1,-0.5) and +(-1,-0.5).. +(2,0);
\draw (-2,3) -- +(1,0);
\draw (1,3) -- +(1,0);
\draw (-1,3).. controls +(1,0.5) and +(-1,0.5).. +(2,0);
\draw (-1,3).. controls +(1,-0.5) and +(-1,-0.5).. +(2,0);
\draw (-2,0) -- +(0,3);
\draw[postaction={decorate}] (-1,0) .. controls +(0,1.5) and +(0,1.5).. +(2,0);
\draw[postaction={decorate}] (1,3) .. controls +(0,-1.5) and +(0,-1.5).. +(-2,0);
\draw (2,0) -- +(0,3);
\fill (0,0) circle (2pt and 2pt);
\end{scope} 
\end{tikzpicture}

\vspace{0.7cm}
\begin{tikzpicture}[scale=0.32]
\begin{scope}[xshift=0cm, yshift= 0cm]
\draw (0,0) -- (2,0);
\draw (2,0) -- (1,1);
\draw (1,1) -- (-1,1); 
\draw (-1,1) -- (0,0);
\draw (0,0) -- +(-0.5,-0.5);
\draw (2,0) -- +(0.7,-0.3);
\draw (1,1) -- +(0.5,0.5); 
\draw (-1,1) -- +(-0.7,0.3);
\draw (0,-4) -- (2,-4);
\draw[dashed] (2,-4) -- (1,-3);
\draw[dashed] (1,-3) -- (-1,-3); 
\draw (-1,-3) -- (-0.5,-3.5);
\draw[dashed] (-0.5,-3.5) -- (0,-4);
\draw (0,-4) -- ++(-0.5,-0.5) -- +(0,4);
\draw (2,-4) -- ++(0.7,-0.3)--+(0,4);
\draw[dashed] (1,-3) -- ++(0.5,0.5)--+(0,3);
\draw (1.5,0.5) -- (1.5,1.5); 
\draw (-1,-3) -- ++(-0.7,0.3)--+(0,4);
\draw (0,-4) -- +(0,4);
\draw (2,-4) -- +(0,4);
\draw[dashed] (1,-3) -- +(0,3);
\draw (1,0)-- +(0,1);
\draw (-1,-3) -- +(0,4);
\end{scope}
\node at (3.9,-2) {$=-$};
\begin{scope}[xshift=7cm, yshift= 0cm, decoration={markings, mark=at
     position 0.5 with {\arrow{>}}},postaction={decorate}]
\fill[gray, opacity =0.5] (-1,1) -- (0,0) arc (180:225:1) -- +(-1,1) arc (225:180:1)-- cycle;
\fill[gray,opacity =0.5] (1,1) -- (2,0) arc (0:-135:1) -- +(-1,1) arc(-135:0:1) --cycle;
\fill[gray,opacity=0.5] (-1,-3) -- (0,-4) arc (180:45:1) -- +(-1,1) arc (45:180:1)-- cycle;
\fill[gray,opacity=0.5] (1,-3) -- (2,-4) arc (0:45:1) -- +(-1,1) arc(45:0:1) --cycle;
\draw (0,0) -- (2,0);
\draw (2,0) -- (1,1);
\draw (1,1) -- (-1,1); 
\draw (-1,1) -- (0,0);
\draw (0,0) -- +(-0.5,-0.5);
\draw (2,0) -- +(0.7,-0.3);
\draw (1,1) -- +(0.5,0.5); 
\draw (-1,1) -- +(-0.7,0.3);
\draw (0,-4) -- (2,-4);
\draw[dashed] (2,-4) -- (1,-3);
\draw[dashed] (1,-3) -- (-1,-3); 
\draw (-1,-3) -- (-0.5,-3.5);
\draw[dashed] (-0.5,-3.5) -- (0,-4);
\draw (0,-4) -- ++(-0.5,-0.5) -- +(0,4);
\draw (2,-4) -- ++(0.7,-0.3)--+(0,4);
\draw[dashed] (1,-3) -- ++(0.15,0.15);
\draw (1,-3) + (0.15,0.15) --++(0.5,0.5) -- +(0,1.7);
\draw[dashed] (1.5,0.5) -- +(0,-1.3);
\draw (1.5,0.5) -- (1.5,1.5); 
\draw (-1,-3) -- ++(-0.7,0.3)--+(0,4);
\draw (0,-4) arc (180:0:1);
\draw (0,0) arc (-180:0:1);
\draw[dashed] (1,1) arc (0:-180:1);
\draw (-1,-3) arc (180:120:1);
\draw[dashed] (1,-3) arc (0:120:1);
\draw (0,-3)++(45:1) -- +(1,-1);
\draw (0,1)++(-135:1) -- +(1,-1);
\end{scope}
\node at (10.9,-2) {$-$};
\begin{scope}[xshift=14cm, yshift= 0cm, decoration={markings, mark=at
     position 0.5 with {\arrow{>}}},postaction={decorate}]
\fill[gray,opacity=0.5] (-1,1) ..controls +(0,-1.5) and +(-0.1,0).. (-0.5,-0.8) -- ++(2,0).. controls +(-0.1,0) and +(0,-1.5) .. (1,1);
\fill[gray,opacity=0.5] (0,0) ..controls +(0,-0.8) and +(0.1,0).. (-0.5,-0.8) -- ++(2,0).. controls +(0.1,0) and +(0,-0.8) .. (2,0);
\fill[gray,opacity=0.5] (-1,-3) ..controls +(0,0.8) and +(-0.1,0).. ++(0.5,0.8) -- ++(2,0).. controls +(-0.1,0) and +(0,0.8) .. (1,-3);
\fill[gray,opacity=0.5] (0,-4) ..controls +(0,1.5) and +(0.1,0).. (-0.5,-2.2) -- ++(2,0).. controls +(0.1,0) and +(0,1.5) .. (2,-4);
\draw (0,0) -- (2,0);
\draw (2,0) -- (1,1);
\draw (1,1) -- (-1,1); 
\draw (-1,1) -- (0,0);
\draw (0,0) -- +(-0.5,-0.5);
\draw (2,0) -- +(0.7,-0.3);
\draw (1,1) -- +(0.5,0.5); 
\draw (-1,1) -- +(-0.7,0.3);
\draw (0,-4) -- (2,-4);
\draw[dashed] (2,-4) -- (1,-3);
\draw[dashed] (1,-3) -- (-1,-3); 
\draw (-1,-3) -- (-0.5,-3.5);
\draw[dashed] (-0.5,-3.5) -- (0,-4);
\draw (0,-4) -- ++(-0.5,-0.5) -- +(0,4);
\draw (2,-4) -- ++(0.7,-0.3)--+(0,4);
\draw[dashed] (1,-3) -- ++(0.5,0.5)--+(0,0.3);
\draw (1.5,-2.2) -- (1.5,-0.8);
\draw[dashed] (1.5,-0.8) -- (1.5,0.5);
\draw (1.5,0.5)-- (1.5,1.5 );
\draw (1.5,0.5) -- (1.5,1.5); 
\draw (-1,-3) -- ++(-0.7,0.3)--+(0,4);
\draw[dashed] (1,-3) ..controls +(0,0.8) and +(-0.1,0) .. ++(0.5,0.8).. controls +(0.1,0) and +(0,1.5) .. (2,-4);
\draw (-1,-3) ..controls +(0,0.8) and +(-0.1,0) .. ++(0.5,0.8).. controls +(0.1,0) and +(0,1.5) .. (0,-4);
\draw (1,1) ..controls +(0,-1.5)  and +(-0.1,0) .. (1.5,-0.8).. controls +(0.1,0) and +(0,-0.8) .. (2,0);
\draw (-1,1) ..controls +(0,-1.5) and +(-0.1,0) .. (-0.5,-.8).. controls +(0.1,0) and +(0,-0.8) .. (0,0);
\draw (-0.5,-0.8)-- +(2,0);
\draw (-0.5,-2.2)-- +(2,0);
\end{scope} 
\end{tikzpicture}

\vspace{0.7cm}
\begin{tikzpicture}[scale=0.5]
\input{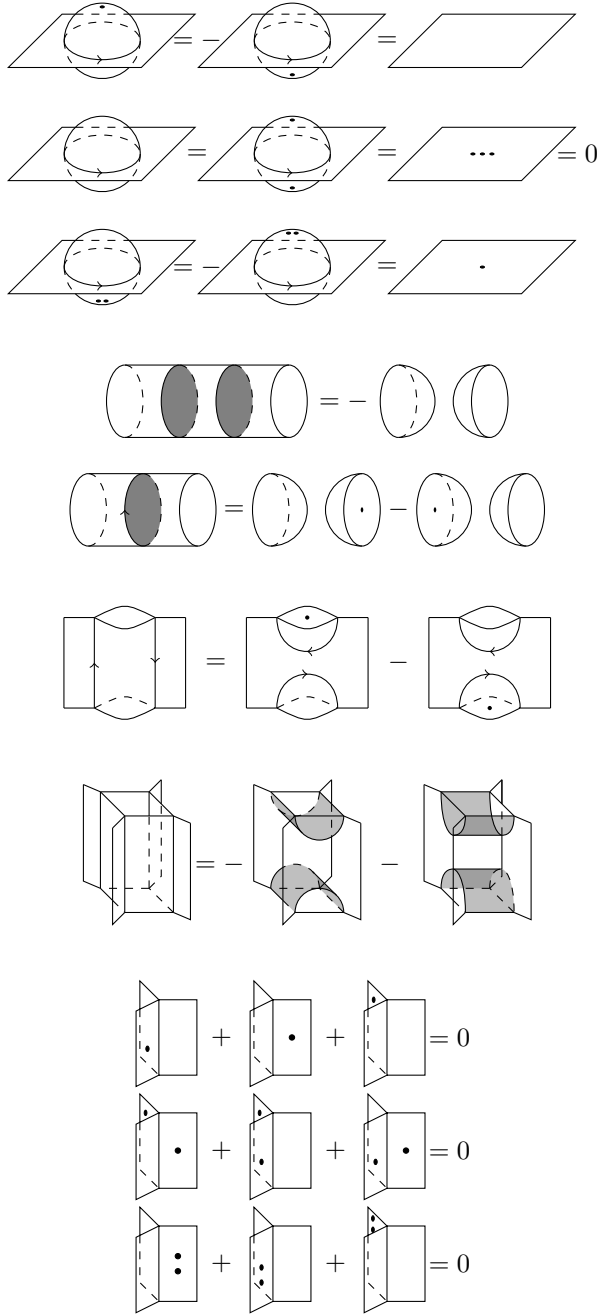} 
\end{tikzpicture}
  \caption{Local relations for 2-morphism in $\mathbb{WT}$. The first 3 lines are called bubbles relations, the 2 next are called bamboo relations, the one after digon relation, then we have the square relation and the 3 last ones are the dots migration relations.}
  \label{fig:localrel}
\end{figure}

This permits to define a new graded 2-category $\mathbb{WT}$. Its objects and its 1-morphisms are the ones of the 2-category $\mathcal{WT}$ while its 2-morphisms-spaces are the ones of $\mathcal{WT}$ modded out like in the last paragraph. One should notice that a $(w_b,w_t)$-foam can always be deformed into a $(\mathrm{tr}(\bar{w_b}w_t),\emptyset)$-foam and vice-versa. Khovanov's results restated in this language give that if $w_b$ and $w_t$ are $(\epsilon_0,\epsilon_1)$-web-tangles, the graded dimension of $\hom_{\mathbb{WT}}(w_t,w_b)$ is given by $\kup{\mathrm{tr}(\bar{w_b}w_t)}\cdot q^{l(\epsilon_0)+l(\epsilon_1)}$. Note that when $\epsilon_1=\emptyset$, there is no need to take the closure, because $w_b\bar{w_t}$ is already a closed web. The shift by $l(\epsilon_0)+l(\epsilon_1)$ comes from the fact that $\chi(\mathrm{tr}(\bar{w_b}w_t)) = \chi(w_t)+\chi(w_b) - (l(\epsilon_0)+l(\epsilon_1))$.

\begin{prop}\label{prop:relFR}
We consider the set \FR{} of local relations which consists of:
\begin{itemize}
\item the surgery relation,
\item the evaluations of the dotted spheres and of the dotted theta-foams,
\item the square relations and the digon relations (see figure~\ref{fig:localrel}).
\end{itemize}
We call them \emph{foam relations} or relations \FR, then for any closed web $w$ $\F(w)$ is isomorphic to $\hom_{\mathcal{WT}}(\emptyset,w)$ modded out by \FR.
\end{prop}

\subsection{The algebras $K^\epsilon$ and the web-modules}
\label{sec:agebras-kepsilon-web}
\subsubsection{The graded version}
\label{sec:Kegraded}

We want to extend the Khovanov TQFT to the 0-dimensional objects \ie to build a 2-functor from the 2-category $\mathcal{WT}$ to the 2-category of algebras. We follow the methodology of \cite{MR1928174} and we start by defining the image of the $0$-objects: they are the algebras $K^\epsilon$. This can be compared with \cite{2012arXiv1206.2118M}.
\begin{dfn}
  Let $\epsilon$ be an admissible finite sequence of signs. We define $\tilde{K}^\epsilon$ to be the full sub-category of $\hom_{\mathbb{WT}}(\emptyset,\epsilon)$ whose objects are element of $\NE(\epsilon)$ (see theorem~\ref{thm:Kuperberg}). This is a graded $\QQ$-algebroid. We recall that a $k$-algebroid is just a $k$-linear category. This can be seen as an algebra by setting:

\[K^\epsilon = \bigoplus_{(w_b,w_t)\in \NE(\epsilon)^2} \hom_{\mathbb{WT}}(w_b,w_t)\]

and the multiplication on $K^\epsilon$ is given by the composition of morphisms in $\tilde{K}_\epsilon$ whenever it's possible and by zero when it's not. We will denote $\tensor[_{w_1}]{K}{^{\epsilon}_{w_2}}\eqdef \hom_{\mathbb{WT}}(w_2,w_1)$. This is a unitary algebra because of theorem~\ref{thm:Kuperberg}. The unite element is $\sum_{w\in \NE(\epsilon)} 1_w$. Suppose $\epsilon$ is fixed, for $w$ a non-elliptic $\epsilon$-web, we define $P_w$ to be the left $K^\epsilon$-module:
\[
P_w=\bigoplus_{w'\in \NE(\epsilon)}\hom_{\mathbb{WT}}(w,w') = \bigoplus_{w'\in \NE(\epsilon)} \tensor[_{w'}]{K}{^\epsilon_{w}}.
\]The structure of module is given by composition on the left.
\end{dfn}
\begin{notation}
The algebra $K^\epsilon$ is defined over $\QQ$. In this paper we only use this version, but it could have been defined over $\ZZ$ or over any field $k$. When we want to focus that we work over $\QQ$ we write $K_\QQ^\epsilon$.
\end{notation}
For a given $\epsilon$, the modules $P_w$ are all projective and we have the following decomposition in the category of left $K^\epsilon$-modules:
\[
K^\epsilon\simeq\bigoplus_{w\in \NE(\epsilon)} P_w.
\]
\begin{prop}\label{prop:kupcategorified}
Let $\epsilon$ be an admissible sequence of signs, and $w_1$ and $w_2$ two non-elliptic $\epsilon$-webs, then the graded dimension of $\hom_{K^\epsilon}(P_{w_1}, P_{w_2})$ is given by $\kup{(\bar{w_1}w_2)}\cdot q^{l(\epsilon)}$.
\end{prop}
\begin{proof}
  An element of  $\hom_{K^\epsilon}(P_{w_1}, P_{w_2})$ is completely determined by the image of $1_{w_1}$ and this element can be sent on any element of $\hom_{\mathbb{WT}}(P_{w_2}, P_{w_1})$, and $\dim_q(\hom_{\mathbb{WT}}(P_{w_1}, P_{w_2}))=\kup{(\bar{w_1}w_2)}\cdot q^{l(\epsilon)}$.
\end{proof}
\begin{req}
  The way we constructed the algebra $K^\epsilon$ is very similar to the construction of $H^n$ in \cite{MR1928174}. Using the same method we can finish the construction of a $0+1+1$ TQFT:
  \begin{itemize}
  \item If $\epsilon$ is an admissible sequence of signs, then $\F(\epsilon) = K^\epsilon$.
  \item If $w$ is a $(\epsilon_1,\epsilon_2)$-web-tangle with $\epsilon_1$ and $\epsilon_2$ admissible, then 
\[\F(w) = \bigoplus_{\substack{u\in \NE(\epsilon_1}) \\ v\in \NE(\epsilon_2)} \F(\bar{u}wv),
\] and it has a structure of graded $K^{\epsilon_1}$-module-$K^{\epsilon_2}$. Note that if $w$ is a non-elliptic $\epsilon$-web, then $\F(w)=P_w$.
\item If $w$ and $w'$ are two $(\epsilon_1,\epsilon_2)$-web-tangles, and $f$ is a $(w,w')$-foam, then we set 
\[
\F(f) = \sum_{\substack{u\in \NE({\epsilon_1}) \\ v\in \NE({\epsilon_2})}} \F(\tensor[_{\bar{u}}]{f}{_v}),
\] where $\tensor[_{\bar{u}}]{f}{_v}$ is the foam $f$ with glued on its sides $\bar{u}\times[0,1]$ and $v\times [0,1]$. This is a map of  graded $K^{\epsilon_1}$-modules-$K^{\epsilon_2}$.
  \end{itemize}
We encourage the reader to have a look at this beautiful construction for the $\sll_2$ case in \cite{MR1928174}. 
\end{req}

In the $\sll_2$ case, the classification of projective indecomposable modules is fairly easy, and a analogous result, would state in our context that the projective indecomposable modules are exactly the modules associated with non-elliptic webs. However we have:

\begin{prop}[\cite{MR2457839}, see \cite{LHRThese} for details]\label{prop:Pwdec}
  Let $\epsilon$ be the sequence of signs:  $(+,-,-,+,+,-,-,+,+,-,-,+)$, and let $w$ and $w_0$ be the two $\epsilon$-webs given by figure~\ref{fig:thewebw}. Then the web-module $P_w$ is decomposable and admits $P_{w_0}$ as a direct factor.
\end{prop}
\begin{figure}[ht]
  \centering
\begin{tikzpicture}[yscale= 0.4, xscale=0.4]
  \input{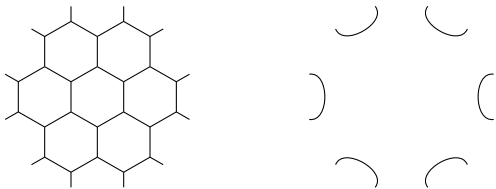}
\end{tikzpicture}
  \caption{The $\epsilon$-webs $w$ (on the left) and $w_0$ (on the right).}
  \label{fig:thewebw}
\end{figure}

\subsubsection{The filtered version}
\label{sec:filtered-version}

The reader interested in this filtered version is advised to have a look at \cite{LewarkThese}, \cite{MR2336253} and \cite{2004math2266G}. 
The filtered version of Khovanov's TQFT is a deformation of the previous one obtained by replacing the Frobenius algebra $\ZZ[X]/(X^3)$ by $\ZZ[X]/(X^3-1)$ which is no longer graded but is filtered. It can be constructed by the universal construction mentioned before. The only thing to change is the evaluation of closed foams. As before, we only need to give is the surgery relation (it is the same as before) and the evaluation of the dotted spheres and theta foams. There are the same expect that the number of dots should be read modulo 3. 
This construction yields a functor $\F'$ from the category of webs to the category of filtered vector space. One can extend it exactly in the same way as we did for $\F$ to a 2-functor from the 2-category of web-tangles to the 2-category of filtered $\QQ$-algebras. We denote by $G^\epsilon$ the filtered algebra $\F'(\epsilon)$ and by $G^\epsilon_\CC$ the same filtered algebra with scalar extended to $\CC$. The $G^\epsilon$-module associated with an $\epsilon$-web $w$ is denoted by $Q_w$ (we keep the same notation for its analogue in $G^\epsilon_\CC-\mathsf{proj}$)

\begin{thm}[{\cite{MR2336253}, \cite[Theorem 1]{2004math2266G} }]
We have the following isomorphism of 2-functor:
\[E \circ \F_\QQ' \simeq \F_\QQ, \]
where $E$ stands for the expected 2-functor from the 2-category of filtered $\QQ$-algebras to the 2-category of graded $\QQ$-algebras.
We have in particular:
\[
E(G^\epsilon) = K^\epsilon \quad \textrm{and} \quad E(Q_w)=P_w.
\]
\end{thm}

To really take advantage of this construction it is important to be able to decompose the polynomial $X^3-1$, this is why will work over $\CC$ though this will have no influence on the ground field in our main result. When forgetting the filtration, $\CC[X]/(X^3-1)$ is isomorphic to $\CC^3$, via the following three mutually orthogonal idempotent elements:
\begin{align*}
  a &= \frac13 (1+X+X^2), \\
  b &= \frac13 (1+jX+j^2X^2), \\
  c &= \frac13 (1+j^2X+jX^2),
\end{align*}
where $j$ is equal to $e^{\frac{2i\pi}3}$. Given an element of $\F_\CC'(w)$, it can be seen as a linear combination of foams. As we have $X^3=1$ we may assume that there are less than 3 dots on each facet, hence we can express any element of $\F_\CC'(w)$ as a linear combination of foam with facets colored by $a$, $b$ and $c$.  One can write some new foam relation \FRp{} for this \emph{colored foams} see figure \ref{fig:FRfilt}. Note in particular the last ones saying that if two adjacent faces have the same color, then the corresponding element in $\F'_\CC(w)$ is equal to 0. A colored foam such that any two adjacent faces have different color is said to be \emph{well-colored}.

\begin{figure}[ht]
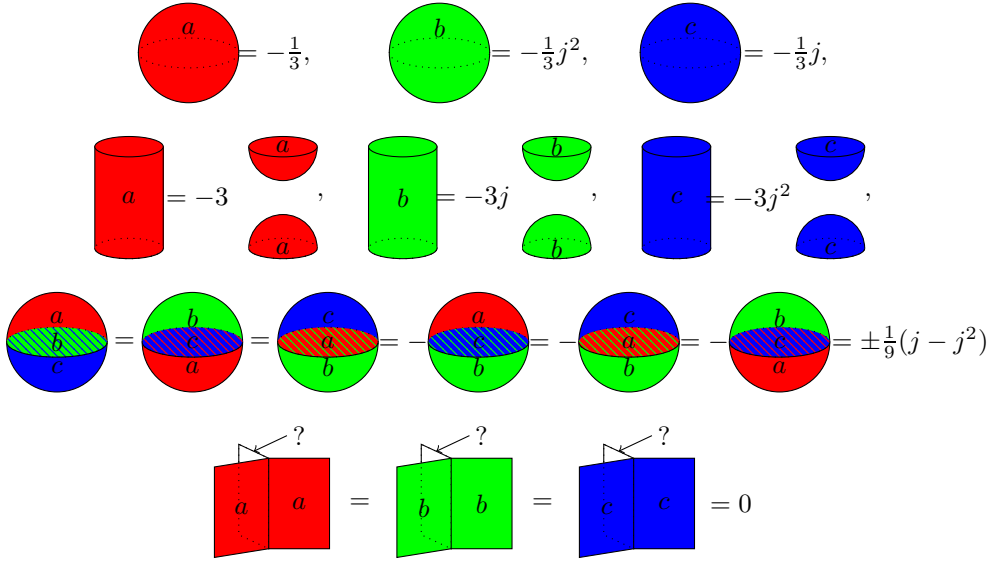

  \centering
  \begin{tikzpicture}[scale= 0.66]
    \input{\imagesfolder/gg_relcoloredfoamsphere}
  \end{tikzpicture} 
\\ \vspace{0.3cm}
  \begin{tikzpicture}[scale= 0.45]
    \input{\imagesfolder/gg_relcoloredfoamtube}
  \end{tikzpicture} 
\\ \vspace{0.3cm}
  \begin{tikzpicture}[scale= 0.66]
    \input{\imagesfolder/gg_relcoloredfoamtheta}
  \end{tikzpicture} 
\\ \vspace{0.3cm}
  \begin{tikzpicture}[yscale=0.6, xscale= 0.8]
    \input{\imagesfolder/gg_relcoloredfoamdoublecol}
  \end{tikzpicture}
  \caption{Relations \FRp{} of colored foams in the filtered context. The question mark indicate that the relations hold whatever the color of the indicated facet is.}
  \label{fig:FRfilt}
\end{figure}

\begin{dfn}
  Let $w_c$ be a colored web,  we denote by $\F_\CC'(w_c)$ the sub-vector space of $\F'_\CC(w)$ spanned by colored foam whose coloring restrict to the coloring $c$ via the correspondence $a\mapsto -1$, $b\mapsto 0$, $c\mapsto 1, $ 
\end{dfn}
 
Rephrasing what was said above we have:

\begin{lem}\label{lem:decincolor}
  Forgetting the filtration, we have the following isomorphism of $\CC$-vector space:
\[ \F'_\CC(w) \simeq \bigoplus_{c\in \col(w)}\F'(w_c).\]
\end{lem}

\begin{lem}[{\cite[Proposition 2.7]{LewarkThese}}]\label{lem:colodim1}
  Let $w_c$ be a colored closed web, then $\F'_\CC(w_c)$ has dimension 1 and any well-colored $w_c$-foam corresponds to a non-zero element in $\F'_\CC(w_c)$. 
\end{lem}

One can extend this colored version of the TQFT to the 2-functor framework: for example, if $w_c$ is a colored $\epsilon$-web such that $c$ restrict to $c'$ on $\epsilon$, one defines:

\[
\F'_\CC(w_c) \eqdef \bigoplus_{\substack{w'\in \NE(\epsilon) \\ c_1 \in \col_{c'}(w')}} \F'_\CC (\overline{w'_{c_1}} w_c)\simeq
\bigoplus_{\substack{w'\in \NE(\epsilon) \\ c_1 \in \col_{c'}(w')}} \CC
\]

The action of $G^\epsilon_\CC$ is defined as usual by composing foams on the left when it is possible (here the compatibility condition involves the colorings), 0 when it is not.
From the previous discussion, we have:
\begin{prop}\label{prop:decQw}
  Let $w$ be an $\epsilon$-web then, forgetting the filtration, we have the following decomposition of $G^\epsilon_\CC$-modules:
\[
Q_w\simeq \bigoplus_{c\in \col(w)} \F'_{\CC}(w_c).
\]
\end{prop}
\begin{req}
  This does \emph{not} hold in the category of filtered modules, the problem to lift a direct sum from the category of unfiltered modules to the category of filtered modules is inspected \cite{MR0364351}. We don't need these refinement here. 
\end{req}


This theorem is the reason why it is interesting to work with the Gornik's version of the algebra:
\begin{thm}[Inspired from {\cite[Theorem 4]{2004math2266G}} ]\label{thm:decGe}
  When forgetting the filtration, the algebra $G^\epsilon$ is semi-simple and has the following decomposition as an algebra:
\[
G_\CC^\epsilon\simeq \bigoplus_{c \in \col(\epsilon)} \mathrm{Mat}_{n(c)}(\CC),
\]
where $n(c)$ is the number of colored non-elliptic webs whose coloring restrict to $c$ on $\epsilon$.
\end{thm}
\marginpar{deal with $\CC$ instead of $\QQ$}
\begin{proof}[Sketch of the proof]
  Thanks to the lemmas~\ref{lem:decincolor} and \ref{lem:colodim1}, we have the following decomposition as a vector space: 
\[ 
G_\CC^\epsilon\simeq \bigoplus_{c \in \col(\epsilon)} \bigoplus_{w^1, w^2 \in \NE(\epsilon)} \bigoplus_{c_i \in \col_\epsilon(w^i)} \CC.
\]
Let $c$ be a coloring of $\epsilon$,
let us denote by $G_c^\epsilon$ the sub-space $ \bigoplus_{w^1, w^2 \in \NE(\epsilon)} \bigoplus_{c_i \in \mathrm{col}_\epsilon(w^i)} \CC$. This clearly forms a sub-algebra of $G^\epsilon$. What remains to show, is that $G_c^\epsilon$ is a matrix algebra for every $c$. The isomorphism can be explicitly constructed:  for any two non-elliptic $\epsilon$-webs $w^1$ and $w^2$ and two colorings $c_1$ and $c_2$ restricting to $c$, we fix one generator of $\F'(\overline{w^1_{c_1}}w^2_{c_2})$ \marginpar{define $\F'$} (when $w^1=w^2$ and $c_1=c_2$, we choose the identity foam, and we make symmetric choices in $(w^1,c_1)$ and $(w^2, c_2)$ ), we call it $\tensor[_{w^1_{c_1}}]{f}{_{w^2_{c_2}}}$. Now define a linear map by sending the matrix $E_{w^1_{c_1}, w^2_{c_2}}$ on the foam:
\[\frac{\tensor[_{w^1_{c_1}}]{f}{_{w^2_{c_2}}}}{\tensor[_{w^1_{c_1}}]{\lambda}{_{w^2_{c_2}}}},\quad \textrm{with $\lambda$ such that:} \quad
\tensor[_{w^2_{c_2}}]{f}{_{w^1_{c_1}}}\tensor[_{w^1_{c_1}}]{f}{_{w^2_{c_2}}} = \left(\tensor[_{w^1_{c_1}}]{\lambda}{_{w^2_{c_2}}}\right)^2 \id_{w^2_{c_2}}. 
 \]
This is routine to check that this is actually an isomorphism of $\CC$-algebras.
\end{proof}
In particular we have proven the following lemma:
\begin{lem}\label{lem:isocoloepsQw}
  Let $w_c$ and $w'_{c'}$ two colored $\epsilon$-webs such that $c$ and $c'$ restrict on the same coloring on $\epsilon$, then $\F'_\CC(w_c) \simeq \F'_{\CC}(w'_{c'})$ as (unfiltered) $G^\epsilon_\CC$-modules.
\end{lem}

\section{Algebraic tools}

\subsection{Hattori-Stallings traces}
\label{sec:hatt-stattl-trac}

  Let $A$ be a finite dimensional $k$-algebra, and $P$ a finitely generated module, then $P$ can be seen as a direct factor of $A^k$ for some $k$. In this context, $\id_P$ can be seen as an idempotent of $Mat_k(A)$. Let us denote $\tr(P)$ the image of the trace of $\id_P$ in $A/[A,A]$.
One can show (see \cite{MR0432781}) that $\tr(P)$ does not depend on the choices, and that\marginpar{cite KAssel ?} $\tr(\cdot)$ is additive with respect to direct sum. Hence we can write this definition:

\begin{dfn}
  Let $A$ a finite dimensional algebra, the map of $\ZZ$-module
\[ 
\begin{array}{rcl}
K_0(A) & \to & A/[A,A] \\
\left[P\right] &\mapsto &\tr(P) 
\end{array}
\] is called the Hattori-Stallings trace.
\end{dfn}

The reason why we introduce the Hattori-Stallings trace in this paper is the following proposition:

\begin{prop}[{\cite[paragraph 1.2 and theorem 1.6]{MR1435369}}] \label{prop:HSTrace}
If characteristic of $k$ is equal to $0$, the map:
\[ 
\begin{array}{rcl}
$k$\otimes_\ZZ K_0(A) & \to & A/[A,A] \\
\lambda\otimes\left[P\right] &\mapsto &\lambda\tr(P) 
\end{array}
\] is injective.
\end{prop}

\subsection{Some facts about graded and filtered algebras and modules} In this subsection, $\ZZ$ is endowed with a structure of $\ZZ[q,q^{-1}]$-module via setting $q=1$.
\label{sec:some-facts-about}
\begin{prop}\label{prop:inTdeg0isenough}
  Let $A$ be a finite dimensional graded algebra, we consider $J(A)$ the Jacobson radical of $A$ and $p$ the projection of $A$ onto $A/([A,A]+J(A))$, then the restriction of $p$ to $A_0$ is surjective.
\end{prop}

\begin{proof}
  As, we mod out by $J(A)$, we may suppose that $A$ is semi-simple and even simple. Thanks to a graded version of Wederburn theorem, one may assume that that $A$ is isomorphic to a matrix algebra with diagonal matrices homogeneous of degree 0 (\cite[theorem 2.10.10]{MR2046303}), and the set of diagonal matrices surjects onto $A/[A,A]$.
\end{proof}

\begin{prop}[\cite{MR2046303}]\label{prop:graded2notgraded}
  Let $A$ be a $\ZZ$-graded algebra and $B$ the same algebra where we forgot the grading. The Grothendieck group $K_0(A)$ of $A$-$\mathsf{proj}$ is naturally endowed with a structure of $\ZZ[q,q^{-1}]$-module while $K_0(B)$ has only a $\ZZ$-module structure. We have the following isomorphism:
\[
K_0(A)\otimes_{\ZZ[q,q^{-1}]}\ZZ \simeq K_0(B).
\]
\end{prop}

\begin{proof}
  The $\ZZ[q,q^-1]$-module $K_0(A)$ is free and admits the isomorphism class of indecomposable projective graded modules as a base\footnote{Of course, one takes only one representant per ``grading shift'' class.}. 
The $\ZZ$-module $K_0(B)$ is free and admits the isomorphism class of indecomposable projective modules as a base. Hence the only thing we need to show is that any projective indecomposable $B$-module is gradable (and therefore becomes a $A$-module), and that two such gradings differ one from the other by an overall shift.

We consider a projective $B$-module $P$. It is a direct factor of $B^n$ for some integer $n$. But $B^n$ is gradable, so that $P$ is gradable. Now if we suppose $P$ to be indecomposable. Denote by $P_1$ and $P_2$ two $A$-modules obtained from $P$ by endowing it with two different gradings. The identity of $P$ gives two mutually inverse morphisms $f$ and $g$ of $A$-modules between $P_1$ and $P_2$. We can write $f=\sum f_i$ a decomposition of $f$ in homogeneous maps . We have $P_2 \simeq \bigoplus f_i(P_1)$. But $P_2$ is indecomposable, so that all the $f_i$'s must be trivial but one, say $i_0$. Hence $P_2$ is equal to $P_1\{i_0\}$.
\end{proof}
The following lemma is not concerned with graded modules however it deals with problems arising when forgetting the gradings.
\begin{lem}\label{lem:ZQqq2Zbasis}
  Let $M$ be a free finitely generated $\ZZ[q,q^{-1}]$-module and $(x_i)_{i\in I}$ a collection of elements of $M$ such that:
  \begin{itemize}
  \item the collection $(1_\QQ\otimes x_i)_{i\in I}$ is a $\QQ[q,q^{-1}]$-base of $\QQ\otimes_\ZZ M$,
  \item the collection $(1_\ZZ \otimes_{\ZZ[q,q^{-1}]}x_i)_{i\in I} $ is a $\ZZ$-base of $\ZZ\otimes_{\ZZ[q,q^{-1}]} M$.
  \end{itemize}
Then $(x_i)_{i\in I}$ is a base of $M$.
\end{lem}
\begin{proof}
  Let $(y_i)_{i\in I}$ be a base of $M$. Then, one can write the matrix of $(x_i)_{i \in I}$ in $(y_i)_{i\in I}$. Seen in  $\QQ\otimes M$ and in  $\ZZ\otimes_{\ZZ[q,q^{-1}]} M$, this matrix is invertible. Its determinant is therefore invertible in ${\QQ[q,q^{-1}]}$  and in $\ZZ$ when $q=1$. Hence, it is equal to $\pm q^k$ for some $k$, therefore this matrix is invertible in $\ZZ[q,q^{-1}]$ and hence $(x_i)_{i\in I}$ is a base of $M$.
\end{proof}
We finish this section with a result about filtered modules and their associated graded modules. 
This theorem is essential for our main result, it permits to use to go from the category of graded module to the category of filtered one, which in our case is easier to deal with.
 We want to point out that the relationship between filtered and graded modules is more complicated that it may first appear. The interested reader may have a look at \cite{MR0364351}.

Let $A$ will be a filtered finite dimensional $k$-algebra. We suppose, furthermore, that its filtration $F_*$ is decreasing ($F_{n+1}(A) \subset F_{n}(A)$) and right-bounded ($F_n$ for $n>n_0$). The modules we consider are finite dimensional over $k$ and their filtrations are, as well, right-bounded\footnote{The finite dimensional and the bounded properties ensure, in the language of \cite{MR0364351}, that the filtrations are discrete (and hence complete and separated) and exhaustive. Note that to fit our context we choose to consider decreasing filtration, while \cite{MR0364351} deals with increasing filtrations.}.
\begin{thm}[{\cite[Theorem 6]{MR0364351}}]\label{thm:sjodin}
  Let $M$ is a projective graded $E(A)$-module, then it can be lifted into a filtered module \ie there exists a module $M'$ such that $E(M')\simeq M$. Further more, if $N$ is a filtered module, then any map $f:M\to E(N)$ can be realized as $E(g)$ for a map $g:M'\to N$.
\end{thm}


\section{Grothendieck group of $K^\epsilon$}

We consider the \emph{categorification map} $\phi$ given by:
\[
\begin{array}{rcl}
  {W^\epsilon} &\longrightarrow & K_0(K^\epsilon_{\QQ}\mathsf{\textrm{-}proj}_{\mathrm{gr}}) \\
  w &\longmapsto & [P_w].
\end{array}
\]
\marginpar{il faut définir $P_w$}
\begin{thm}\label{thm:mainthm}
  The map $\phi$ is an isomorphism of $\ZZ[q, q^{-1}]$-modules.
\end{thm}
The are two key arguments: 
\begin{itemize}
\item the proposition \ref{prop:spanTe} read with algebraic tools will give that the map $\phi$ induces an isomorphism when taking the tensor product with $\QQ$ (see corollary \ref{cor:isoQ}). 
\item the proposition~\ref{prop:ZZqqbase} will do the transition from $\QQ$ to $\ZZ$.
  
\end{itemize}

\subsection{Foams and traces}
\label{sec:foams-traces}

\begin{dfn}\label{dfn:Tespilon}
  Let $T(\epsilon)$ be the $\ZZ$-module spanned by foams in the cylinder $D^2\times \S^1$  with boundary $\epsilon\times \S^1$ modulo foam relations \FR{} (we only allow foam relations inside the solid torus). We could do a universal construction for the torus (because for $\epsilon$ admissible, $\epsilon\times \S^1$ and $\emptyset$ are cobordant in the solid torus)
\end{dfn}

\begin{dfn}\label{dfn:closure}
  Let $w$ be an $\epsilon$-web and $f$ a $(w,w)$-foam. The \emph{closure of $f$}, denoted by $t(f)$ is the element of $T(\epsilon)$ obtained from $f$ by gluing the two ends of $f$ along $w$. This notion extends linearly.
\end{dfn}

We introduced these objects because of the following proposition:

\begin{prop}\label{prop:isoTe}
  The closure induced an isomorphism of $\ZZ$-module between $T(\epsilon)$ and $K^\epsilon_{\ZZ}\otimes_{(K^\epsilon_{\ZZ})^{\textrm{e}}} K^\epsilon_{\ZZ}$, where $(K^\epsilon_{\ZZ})^{\textrm{e}}$ is the envelopping algebra of $K^\epsilon_{\ZZ}$ \ie{}$K^\epsilon_{\ZZ} \otimes_\ZZ(K^\epsilon_{\ZZ})^{\textrm{op}}$.
\end{prop}

\begin{proof}
  Let us first detail the definition of the isomorphism. Setting $t(f)= 0$ for any $(w_1,w_2)$-foam with $w_1\neq w_2$,  the map $t$ of definition \ref{dfn:closure} gives us a map from $K^\epsilon_{\ZZ}$ to $T(\epsilon)$. It is clear that $[K^\epsilon_{\ZZ}, K^\epsilon_{\ZZ}]$ is included in the kernel of $t$, therefore it induces the map $\tilde{t}$ we are looking for. The map $\tilde{t}$ is obviously onto, so what remains to show is its injectivity. It is clear because the foam relations \FR{} used to defined $T(\epsilon)$ are local (lies in balls embedded in the solid torus), hence already exist in the foam description of $K^\epsilon_{\ZZ}$.
\end{proof}

\begin{dfn}
We say that an $\epsilon$-web $w$ contains a $\lambda$ (resp.~a $\cap$, resp.~an $H$) in position $i$ if next to $\epsilon_i$ and $\epsilon_{i+1}$, the $\epsilon$-web $w$ looks like one of the pictures of figure~\ref{fig:dfnUHY}.
\begin{figure}[ht]
  \centering
  \begin{tikzpicture}[xscale = 0.7, yscale=0.5]
    \input{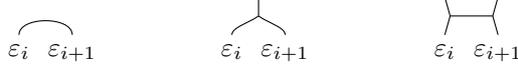}
  \end{tikzpicture}
  \caption{From left to right: a $\lambda$, a $\cap$ and a $H$.}
  \label{fig:dfnUHY}
\end{figure}
\end{dfn}
\begin{notation}
Let $w$ be an $\epsilon$-web $w$ containing a $\cap$ (\resp{}a $\lambda$) in position $i$. We denote by $\epsilon'$ the sequence of signs obtained from $\epsilon$ by removing $\epsilon_i$ and $\epsilon_{i+1}$ (\resp{}by replacing $\epsilon_i$ by $-\epsilon_i$ and by removing $\epsilon_{i+1}$). We denote by $w'$ the $\epsilon'$-web obtained from $w$ by removing the $\cap$ (\resp{}by replacing the $\lambda$ by a single strand).
Let $w$ be an $\epsilon$-web $w$ containing an $H$ in position $i$. We denote by $\epsilon'$ (\resp{}$\epsilon''$) the sequence of signs obtained from $\epsilon$ by replacing $\epsilon_i$ and $\epsilon_{i+1}$ by their inverses (\resp{}by removing $\epsilon_i$ and  $\epsilon_{i+1}$). We denote by $w_{||}$ the $\epsilon'$-web  obtained from $w$ by replacing the $H$ by two vertical strands and by $w_{\_}$ the $\epsilon''$-web obtained from $w$ by replacing the $H$ by a $\cup$ (\ie by removing the $H$ and joining together the two ``new'' ends of $w$). 
\end{notation}

\begin{lem}\label{lem:UHYinNE}
Every non-elliptic $\epsilon$-web contains at least a $\lambda$, a $\cap$ or an $H$.
\end{lem}
\begin{proof}
  It follows immediately from proposition~\ref{prop:closed2elliptic}.
\end{proof}

\begin{lem}\label{lem:tracecircle}
  Let $w$ be an $\epsilon$-web with a $\cap$ in position $i$, and $f$ a $(w,w)$-foam then $t(f)$ is equal to a sum (with coefficients) of:
  \begin{enumerate}
  \item a disjoint union of $t(g)$ and $\cap \times \S^1$ where $g$ is
    a (sum of) $(w',w')$-foam(s),
  \item a disjoint union of $t(h)$ and $\cap \times \S^1$ with one dot where $h$ is
    a (sum of) $(w',w')$-foam(s),
  \item a disjoint union of $t(j)$ and $\cap \times \S^1$ with two dots where $j$ is
    a (sum of) $(w',w')$-foam(s).
  \end{enumerate}
  The local aspects of $t(g)$, $t(h)$ and $t(j)$ are depicted on figure~\ref{fig:tracecircle}.
  \begin{figure}[ht]
    \centering
    \begin{tikzpicture}[scale =0.65]
      \begin{scope}[yscale= -1]
  \draw (0,0) ellipse (2cm and 0.5cm);
  \draw (0,0.1) arc (90:108: 2.5cm);
  \draw (0,0.1) arc (90:72: 2.5cm);
  \draw (0,-0.1) arc (-90:-104: 3.5cm);
  \draw (0,-0.1) arc (-90:-76: 3.5cm);
  \draw (-2,0) .. controls +(0,-1.2) and +(0,-1.2) .. (2,0);
  \draw[dotted] (-2,0) arc (180:360:0.62 and 0.5);
  \draw[dotted] (2,0) arc (360:180: 0.62 and 0.5);
  \draw[very thin, <-] (0,0.5) -- (-0.1, 0.75) -- (-2.1, 0.75);
  \draw[very thin, <-] (0,0.1) -- (-0.1, 0.25) -- (-2.1, 0.25);
  \node[left] at (-2.15, 0.25) {$\epsilon_i\times \mathbb{S}^1$};
  \node[left] at (-2.15, 0.75) {$\epsilon_{i+1}\times \mathbb{S}^1$};
 \end{scope}

\begin{scope}[yscale= -1, xshift = 5 cm]
  \draw (0,0) ellipse (2cm and 0.5cm);
  \draw (0,0.1) arc (90:108: 2.5cm);
  \draw (0,0.1) arc (90:72: 2.5cm);
  \draw (0,-0.1) arc (-90:-104: 3.5cm);
  \draw (0,-0.1) arc (-90:-76: 3.5cm);
  \draw (-2,0) .. controls +(0,-1.2) and +(0,-1.2) .. (2,0);
  \draw[dotted] (-2,0) arc (180:360:0.62 and 0.5);
  \draw[dotted] (2,0) arc (360:180: 0.62 and 0.5);
  \fill (0, -0.3) ellipse (2pt and 1.5pt);
 \end{scope}

\begin{scope}[yscale= -1, xshift = 10cm]
  \draw (0,0) ellipse (2cm and 0.5cm);
  \draw (0,0.1) arc (90:108: 2.5cm);
  \draw (0,0.1) arc (90:72: 2.5cm);
  \draw (0,-0.1) arc (-90:-104: 3.5cm);
  \draw (0,-0.1) arc (-90:-76: 3.5cm);
  \draw (-2,0) .. controls +(0,-1.2) and +(0,-1.2) .. (2,0);
  \draw[dotted] (-2,0) arc (180:360:0.62 and 0.5);
  \draw[dotted] (2,0) arc (360:180: 0.62 and 0.5);
  \fill (-0.2, -0.3) ellipse (2pt and 1.5pt);
  \fill (+0.2, -0.3) ellipse (2pt and 1.5pt);
 \end{scope}
    \end{tikzpicture}
    \caption{The foams $t(g)$, $t(h)$ and $t(j)$.}
    \label{fig:tracecircle}
  \end{figure}
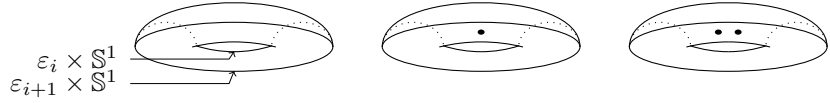
\end{lem}

\begin{proof}
  Thanks to the surgery relation we can cut $f$ next to the circle constituted of the $\cap$ of $w$ and its mirror image.  After taking the closure this gives us the wanted expression for $t(f)$.
\end{proof}

\begin{lem}\label{lem:tracedigon}
  Let $w$ be an $\epsilon$-web with a $\lambda$ in position $i$, and $f$ a $(w,w)$-foam then $t(f)$ is equal to a sum (with coefficients) of:
  \begin{enumerate}
  \item $t(g)$ glued together with $\lambda \times \S^1$ where $g$ is
    a (sum of) $(w',w')$-foam(s),
  \item $t(h)$ glued with $\lambda \times S^1$ with a dot next to $\epsilon_i \times \S^1$ where $h$ is a (sum of) $(w',w')$-foam(s). 
  \end{enumerate}
\end{lem}

\begin{proof}
  Thanks to the digon relation we can write $f$ as a sum of two $(w,w)$-foams $f_1$ and $f_2$ with (see figure \ref{fig:tracedigon}), the foam $f_1$ corresponds to the first term, the foam $f_2$ to the second.
  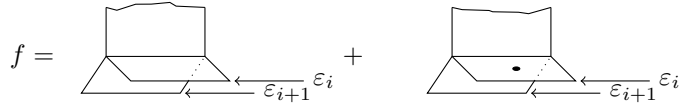
\begin{figure}[ht]
    \centering
    \begin{tikzpicture}[scale = 0.65]
      \node at (-1.5, 0) {$f = $};

\begin{scope}[decoration={random steps, segment length=2mm}]
  \draw (0,1) -- (0,0) -- (2,0) --  (2,1);
  \draw[thin, decorate] (0,1) -- (2,1);
  \draw (2,0) -- (2.5,-0.5) --(0.5, -0.5) -- (0,0);
  \draw (0,0) -- (-0.5, -0.75) -- (1.5, -0.75) -- (1.666 ,-0.5);
  \draw[dotted] (1.666, -0.5) -- (2,0); 
  \node[right] at (3, -0.75) {$\epsilon_{i+1}$};
  \node[right] at (4, -0.5) {$\epsilon_{i}$};
  \draw[<-, thin] (1.6, -0.75) -- (3, -0.75);
  \draw[<-, thin] (2.6, -0.5) -- (4, -0.5);
\end{scope}
\node at (5, 0) {$+$};
\begin{scope}[xshift= 7cm, decoration={random steps, segment length=2mm}]
  \draw (0,1) -- (0,0) -- (2,0) --  (2,1);
  \draw[thin, decorate] (0,1) -- (2,1);
  \draw (2,0) -- (2.5,-0.5) --(0.5, -0.5) -- (0,0);
  \draw (0,0) -- (-0.5, -0.75) -- (1.5, -0.75) -- (1.666 ,-0.5);
  \draw[dotted] (1.666, -0.5) -- (2,0); 
  \filldraw (1.3, -0.25) ellipse (2pt and 1pt);
  \node[right] at (3, -0.75) {$\epsilon_{i+1}$};
  \node[right] at (4, -0.5) {$\epsilon_{i}$};
  \draw[<-, thin] (1.6, -0.75) -- (3, -0.75);
  \draw[<-, thin] (2.6, -0.5) -- (4, -0.5);
\end{scope}
    \end{tikzpicture}
    \caption{The first term in the sum is $f_1$, the second is $f_2$.}
    \label{fig:tracedigon}
  \end{figure}

\end{proof}

\begin{lem}\label{lem:tracequare}
  Let $w$ be an $\epsilon$-web with an $H$ in position $i$, and $f$ a $(w,w)$-foam, then $t(f)$ if equal to a sum (with coefficients) of:
\begin{enumerate}
\item $t(g)$ glued together with $H\times \S^1$ along ($(\epsilon'_i\cup \epsilon'_{i+1}) \times \S^1$) where $g$ is a (sum of) $(w_{||}, w_{||})$-foam(s).
\item $t(h_1) \cup (\cap_i \times \S^1)$ where $h_1$ is a (sum of) $(w_{\_}, w_{\_})$-foam(s).
\item  $t(h_2) \cup (\cap_i \times \S^1)$ with a dot on $\cap_i\times \S^1$ where $h_2$ is a (sum of) $(w_{\_}, w_{\_})$-foam(s).
\end{enumerate}
\end{lem}

\begin{proof}
This is a direct application of the square relation: the foam $f$ can be expressed as a sum of two foams $f_1$ and $f_2$ which around $(\epsilon_i \cup \epsilon_{i+1}) \times I$ have the shapes given by figure~\ref{fig:tracesquare}. 
\begin{figure}[ht]
  \centering
  \begin{tikzpicture}[xscale=0.7, yscale= 0.4]
    \begin{scope}[xshift=0cm, yscale= 1= 0cm, rotate = -90]
\draw (0,0) -- (2,0);
\draw (2,0) -- (1,1);
\draw (1,1) -- (-1,1); 
\draw (2,0) -- +(0.7,-0.3);
\draw (1,1) -- +(0.5,0.5); 
\draw (0,-4) -- (2,-4);
\draw[densely dotted] (2,-4) -- (1,-3);
\draw[densely dotted] (1,-3) -- (0,-3); 
\draw (-1,-3) -- (0,-3); 
\draw (2,-4) -- ++(0.7,-0.3)--+(0,4);
\draw[densely dotted] (1,-3) -- ++(0.5,0.5)--+(0,3);
\draw (1.5,0.5) -- (1.5,1.5); 
\draw (0,-4) -- +(0,4);
\draw (2,-4) -- +(0,4);
\draw[densely dotted] (1,-3) -- +(0,3);
\draw (1,0)-- +(0,1);
\draw (-1,-3) -- +(0,4);
\draw[<-, very thin] (2.7 ,-2.3) -- (3.3,-2.3);
\draw[<-, very thin] (1.5 ,1) -- (3.3,1);
\node[below] at (3.3, -2.3) {$\epsilon_i$};
\node[below] at (3.3, 1) {$\epsilon_{i+1}$};
\end{scope}
\begin{scope}[xshift=9cm, yscale = 1, decoration={markings, mark=at
     position 0.5 with {\arrow{>}}},postaction={decorate}, rotate = -90]
\fill[gray,opacity=0.5] (-1,1) ..controls +(0,-1.5) and +(-0.1,0).. (-0.5,-0.8) -- ++(2,0).. controls +(-0.1,0) and +(0,-1.5) .. (1,1);
\fill[gray,opacity=0.5] (0,0) ..controls +(0,-0.8) and +(0.1,0).. (-0.5,-0.8) -- ++(2,0).. controls +(0.1,0) and +(0,-0.8) .. (2,0);
\fill[gray,opacity=0.5] (-1,-3) ..controls +(0,0.8) and +(-0.1,0).. ++(0.5,0.8) -- ++(2,0).. controls +(-0.1,0) and +(0,0.8) .. (1,-3);
\fill[gray,opacity=0.5] (0,-4) ..controls +(0,1.5) and +(0.1,0).. (-0.5,-2.2) -- ++(2,0).. controls +(0.1,0) and +(0,1.5) .. (2,-4);
\draw (0,0) -- (2,0);
\draw (2,0) -- (1,1);
\draw (1,1) -- (-1,1); 
\draw (2,0) -- +(0.7,-0.3);
\draw (1,1) -- +(0.5,0.5); 
\draw (0,-4) -- (2,-4);
\draw[densely dotted] (2,-4) -- (1,-3);
\draw[densely dotted] (1,-3) -- (-1,-3); 
\draw(-0.1,-3) -- (-1,-3); 
\draw (2,-4) -- ++(0.7,-0.3)--+(0,4);
\draw[densely dotted] (1,-3) -- ++(0.5,0.5)--+(0,0.3);
\draw[densely dotted] (1.5,-2.2) -- (1.5,0.5);
\draw (1.5,0.5)-- (1.5,1.5 );
\draw[densely dotted] (1,-3) ..controls +(0,0.8) and +(-0.1,0) .. ++(0.5,0.8).. controls +(0.1,0) and +(0,1.5) .. (2,-4);
\draw (-1,-3) ..controls +(0,0.8) and +(-0.1,0) .. ++(0.5,0.8).. controls +(0.1,0) and +(0,1.5) .. (0,-4);
\draw (1,1) ..controls +(0,-1.5)  and +(-0.1,0) .. (1.5,-0.8).. controls +(0.1,0) and +(0,-0.8) .. (2,0);
\draw (-1,1) ..controls +(0,-1.5) and +(-0.1,0) .. (-0.5,-.8).. controls +(0.1,0) and +(0,-0.8) .. (0,0);
\draw (-0.5,-0.8)-- +(2,0);
\draw (-0.5,-2.2)-- +(2,0);
\draw[<-, very thin] (2.7 ,-2.3) -- (3.3,-2.3);
\draw[<-, very thin] (1.5 ,1) -- (3.3,1);
\node[below] at (3.3, -2.3) {$\epsilon_i$};
\node[below] at (3.3, 1) {$\epsilon_{i+1}$};
\end{scope}
  \end{tikzpicture}
  \caption{The local aspects of $f_1$ (on the left) and of $f_2$ (on the right). }
  \label{fig:tracesquare}
\end{figure}
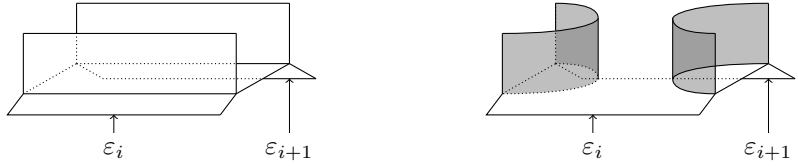

The foam $f_1$ can be obtained by gluing an $H\times I$ to  $(w_{||}, w_{||})$-foam a along $(\epsilon'_i \cup \epsilon'_{i+1})\times I$. This gives the $t(g)$.
Around $(\epsilon_i \cup \epsilon_{i+1}) \times \S^1$ the foam $t(f_2)$ looks like the figure~\ref{fig:tracesquare2}.
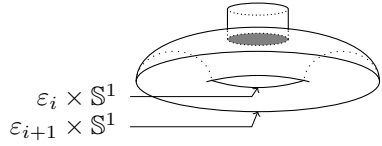
\begin{figure}[ht]
  \centering
  \begin{tikzpicture}[scale= 0.8]
    \begin{scope}[yscale=- 1]
  \draw (0,0) ellipse (2cm and 0.5cm);
  \draw (0,0.1) arc (90:108: 2.5cm);
  \draw (0,0.1) arc (90:72: 2.5cm);
  \draw (0,-0.1) arc (-90:-104: 3.5cm);
  \draw (0,-0.1) arc (-90:-76: 3.5cm);
  \draw (-2,0) .. controls +(0,-1.2) and +(0,-1.2) .. (2,0);
  \draw[dotted] (-2,0) arc (180:360:0.62 and 0.5);
  \draw[dotted] (2,0) arc (360:180: 0.62 and 0.5);
  \draw[very thin, <-] (0,0.5) -- (-0.1, 0.75) -- (-2.1, 0.75);
  \draw[very thin, <-] (0,0.1) -- (-0.1, 0.25) -- (-2.1, 0.25);
  \node[left] at (-2.15, 0.25) {$\epsilon_i\times \mathbb{S}^1$};
  \node[left] at (-2.15, 0.75) {$\epsilon_{i+1}\times  \mathbb{S}^1$};
  \filldraw[densely dotted, fill= gray] (0,-0.7) ellipse (0.5 and 0.1);
  \draw[densely dotted] (-0.5,-0.7) -- (-0.5, -0.85);
  \draw[densely dotted] (0.5,-0.7) -- (0.5, -0.85);
  \draw (-0.5,-1.2) -- (-0.5, -0.85);
  \draw (0.5,-1.2) -- (0.5, -0.85);
  \draw[densely dotted] (-0.5, -1.2) arc (180:0:0.5 and 0.1);
  \draw (-0.5, -1.2) arc (-180:0:0.5 and 0.1);
 \end{scope}
  \end{tikzpicture}
  \caption{The foam $t(f_2)$. It is a half torus glued with a half-sphere, the gluing zone is gray on this picture.}
  \label{fig:tracesquare2}
\end{figure}
 We can see that there is a singular circle which does not wing around the solid tori, and which bounds a disk. Hence we can perform a bamboo relation (see figure~\ref{fig:localrel}). We end up with a sum of two terms corresponding to $t(h_1)$ and $t(h_2)$.
\end{proof}

\begin{prop}\label{prop:spanTe}
  The module $T(\epsilon)$ is spanned by $w\times \S^1$, with some dots, for $w$ non-elliptic.
\end{prop}
\begin{proof}This is a an easy consequence of the lemmas \ref{lem:UHYinNE}, \ref{lem:tracecircle}, \ref{lem:tracedigon} and \ref{lem:tracequare}.
\end{proof}

  



\subsection{Counting dimensions}
\label{sec:counting-dimension} Let us write  $W'^\epsilon$ for $\ZZ\otimes_{\ZZ[q,q^{-1}]}W^\epsilon$.
The aim of this subsection is to show the following proposition:
\marginpar{définir $\phi'$}
\begin{prop}\label{prop:isoQ}
  The map $\tilde{\phi'}\eqdef\id_\QQ\otimes\phi': \QQ\otimes W'^\epsilon \to \QQ \otimes K_0(K^\epsilon_{\QQ}\mathsf{\textrm{-}proj})$ is an isomorphism.
\end{prop}

\begin{cor}\label{cor:isoQ}
  The map $\tilde{\phi}\eqdef\id_\QQ\otimes\phi: \QQ\otimes W^\epsilon \to \QQ \otimes K_0(K^\epsilon_{\QQ}\mathsf{\textrm{-}proj}_{\mathrm{gr}})$ is an isomorphism.
\end{cor}
\begin{proof}[Proof of \ref{cor:isoQ} from \ref{prop:isoQ}]
Thanks to proposition~\ref{prop:graded2notgraded} we have $K_0(K^\epsilon_{\QQ}\mathsf{\textrm{-}proj}_{\mathrm{gr}}) \simeq \ZZ[q,q^{-1}]\otimes K_0(K^\epsilon_{\QQ}\mathsf{\textrm{-}proj})$, hence with the natural identification we have
$\tilde{\phi}=\id_{\ZZ[q,q^{-1}]}\otimes \tilde{\phi'}$
\end{proof}

\begin{proof}
Let us denote $\psi =\tilde{\phi}$. This is a map of $\QQ$-vector spaces, hence it is enough to show that:
\begin{itemize}
\item  the map $\psi$ is injective, (this is the lemma~\ref{lem:lem1});
\item we have $\dim \QQ\otimes K_0(K^\epsilon_{\QQ}\mathsf{\textrm{-}proj}) \leq \dim \QQ\otimes W^\epsilon$ (this is the lemma~\ref{lem:lem2}).
\end{itemize}
\begin{lem}\label{lem:lem1}
  The map $\psi$ is injective.
\end{lem}
\begin{proof}
  Thanks to proposition~\ref{prop:kupcategorified},  the pairing induced by the Kuperberg bracket is transported via $\tilde{\phi}$ to the graded dimension of hom-spaces. Note that here we work with $\psi$ so that we are interested in this statement through the evaluation $q\mapsto 1$.
The pairing induced by the Kuperberg bracket (evaluated in $1$) is non degenerate (indeed, under natural identification, this is the pairing between $\hom_{U{\sll_3}}(\CC, V^{\otimes\epsilon})$ and $\hom_{U(\sll_3)}(V^{\otimes\epsilon},\CC)$, and the non-elliptic webs are known to be a basis of these spaces (theorem~\ref{thm:Kuperberg}), so that $\psi$ is an isometry and hence is injective.
\end{proof}
\begin{lem}\label{lem:lem2}
The following inequality holds:
 \[\dim \QQ\otimes K_0(K^\epsilon_{\QQ}\mathsf{\textrm{-}proj}) \leq \dim \QQ\otimes W^\epsilon.\]
\end{lem}
\begin{proof}
  From the proposition~\ref{prop:HSTrace}, we have: $\dim \QQ\otimes K_0(K^\epsilon_{\QQ}\mathsf{\textrm{-}proj})\leq \dim T(K^\epsilon_{\QQ})$. \marginpar{gérer q -> 1 ... BLABLA} Thanks to lemma~\ref{prop:isoTe}, we have the projection $\pi: T(\epsilon) \to T(K^\epsilon_{\QQ}) \simeq T(\epsilon) /J((K^\epsilon_{\QQ})$. The proposition~\ref{prop:inTdeg0isenough} says implies that when restricting $\pi$ to $T_0(\epsilon)$ the subspace of $T(\epsilon)$ spanned by foams of degree 0, the application remains surjective. The collection of element given in proposition~\ref{prop:spanTe} are all homogeneous, so that $T_0(\epsilon)$ is spanned by the elements among this collection which have degree 0. These are exactly the foams $w\times \mathbb{S}_1$ for $w$ a non-elliptic web. So that we have:
\[
\dim T(K^\epsilon_{\QQ}) \leq \dim \QQ\otimes T_0(\epsilon) \leq  \dim \QQ\otimes W^\epsilon,
\]
and finally $ \dim \QQ\otimes K_0(K^\epsilon_{\QQ}\mathsf{\textrm{-}proj}) \leq \dim \QQ\otimes W^\epsilon$
\end{proof}
Putting the two lemmas together concludes the proof of proposition~\ref{prop:isoQ}. Of course the inequalities in lemma~\ref{lem:lem2} and its poof are all equalities.
\end{proof}


\subsection{Finding the $\ZZ$-base}
\label{sec:finding-base}

The corollary~\ref{cor:isoQ} tells that the (isomorphism classes of) web-modules are a $\QQ[q,q^{-1}]$-base of $\QQ\otimes K_0(K^\epsilon-\mathsf{prof}_{\mathrm{gr}})$. To prove the theorem~\ref{thm:mainthm}, it remains to show that they are actually a $\ZZ[q,q^{-1}]$-base of $K_0(K^\epsilon\mathsf{\textrm{-}proj}_{\mathrm{gr}})$.
\begin{prop} \label{prop:ZZqqbase}
The collection $([P_w])_{w\in \NE(w)}$ is a $\ZZ[q,q^{-1}]$ base of $K_0(K^\epsilon\mathsf{\textrm{-}proj}_{\mathrm{gr}})$.
\end{prop}
Thanks to lemma~\ref{lem:ZQqq2Zbasis}, this will follow from the following lemma:
\begin{lem}\label{lem:ZZbase}
  The collection $([P_w])_{w\in \NE(w)}$ is a $\ZZ$ base of $\ZZ\otimes_{\ZZ[q,q^{-1}]}K_0(K^\epsilon\mathsf{\textrm{-}proj}_{\mathrm{gr}})$.
\end{lem}
\begin{proof}
  Let $M$ be an (ungraded) projective $K^\epsilon$-module, thanks to \ref{prop:isoQ} we can write:
\[
[M] =  \sum_{w} \lambda_w [P_w],
\]
with $(\lambda_w)_{w\in \NE(\epsilon)}$ a collection of coefficient in $\QQ$. As $M$ is projective it is gradable and hence thanks to theorem~\ref{thm:sjodin}, we can lift it in the category of filtered modules over $G^\epsilon$ this give us a module $M'$, each of the module $P_w$ is isomorphic to $E(Q_w)$, so that, after forgetting the filtration, we have, in $K_0(G_\QQ^\epsilon)$:
\[ [M'] =  \sum_{w} \lambda_w [Q_w].\]
Tensoring the modules with $\CC$, the theorem~\ref{thm:decGe} gives this equality in $K_0(G_\CC^\epsilon)$:
\[
[M']= \sum_{c \in \textrm{col}(\epsilon)} \mu_c [G_c],
\]
where $\mu_c$ are non-negative integers.
Let us denote by $J$ the set of webs $w$ such that $\lambda_w$ is not in $\ZZ$. Suppose that this set is not empty. According to the lemma~\ref{lem:uniquecoloring} we can find a coloring $c_0$ such that there is only one web $w_{i_0}$ in $J$ and one coloring $c_{i_0}$ of $w_0$ such that $c_{i_0}$ restrict to $c_0$. But this implies that $\mu_{c_0}$ is not an integer, this is contradictory. Hence all the $\lambda$'s lie in $\ZZ$. This finish the proof of lemma~\ref{lem:ZZbase}.
\end{proof}
\subsection{Further questions}
\label{sec:further-questions}
The discussion about traces made in section \ref{sec:foams-traces}, is nothing but a computation of the 0\textsuperscript{th} group of the Hochshild homology of $K^\epsilon$. It would be interesting to compute the other homology groups. 

Even if we found a base of $K_0(K^\epsilon\mathsf{\textrm{-}proj}_{\mathrm{gr}})$, this does not give a complete collection of indecomposable projective graded modules \ie a positive base for $K_0(K^\epsilon\mathsf{\textrm{-}proj}_{\mathrm{gr}})$. It has been investigated in~\cite{LHR2}, but this classification remains completely open, such a collection would decategorify on a dual canonical base.

\bibliographystyle{alpha}
\bibliography{../../biblio}

\end{document}